\documentclass[a4paper,11pt,reqno]{amsart}
\usepackage{xcolor}

\usepackage{fourier}

\usepackage{color}

\usepackage{amsmath}
\usepackage{amssymb}
\usepackage{amsthm}
\usepackage{amsfonts}
\usepackage{mathtools}
\usepackage{enumitem}
\usepackage{cite}
\usepackage{hyperref}

\newtheorem{theorem}{Theorem}[section]
\newtheorem{lemma}[theorem]{Lemma}

\newtheorem{prop}[theorem]{Proposition}

\newtheorem{definition}{Definition}

\newtheorem{remark}{Remark}

\newcommand{\R}{\mathbb{R}}

\newcommand{\I}{\mathcal{I}}
\newcommand{\T}{\mathcal{T}}

\newcommand{\1}{\mathbf{1}}
\newcommand{\eps}{\varepsilon}
\newcommand{\ovl}{\overline}

\numberwithin{equation}{section}

\begin{document}

\title[]{Comparison principle for general nonlocal Hamilton-Jacobi equations with superlinear gradient}

\author[]{Adina Ciomaga}
\address{Adina Ciomaga: 
Université Paris Cité, 
CNRS, 
Sorbonne Université, 
Laboratoire Jacques-Louis Lions (LJLL), 
F-75006 Paris, France
\&
O. Mayer Mathematics Institute, 
Romanian Academy, Ia\c si,
700506 Ia\c si, Romania.}
\email{\tt adina@math.univ-paris-diderot.fr, adina.ciomaga@acadiasi.ro}

\author[]{Tr\'i Minh L\^e}
\address{
Tr\'i Minh L\^e: Institut für Stochastik und Wirtschaftsmathematik, VADOR E105-04, TU Wien, Wiedner Hauptstraße 8, A-1040 Wien.
}
\email{\tt minh.le@tuwien.ac.at}

\author[]{Olivier Ley}
\address{
Olivier Ley: Univ Rennes, INSA Rennes, CNRS, IRMAR - UMR 6625, F-35000 Rennes, France.}
\email{\tt olivier.ley@insa-rennes.fr}

\author[]{Erwin Topp}
\address{
Erwin Topp: Instituto de Matem\'atica, Universidade Federal do Rio de Janeiro, Rio de Janeiro - RJ, 21941-909, Brazil, and Departamento de Matem\'atica y C.C., Universidad de Santiago de Chile,
Casilla 307, Santiago, Chile.}
\email{\tt etopp@im.ufrj.br}

\date{\today}

\maketitle
\textbf{Abstract}.
We obtain the comparison principle for discontinuous viscosity sub- and supersolutions of nonlocal Hamilton-Jacobi equations, with superlinear and coercive gradient terms. The nonlocal terms are integro-differential operators in Lévy form, with general measures: $x$-dependent, possibly degenerate and without any restriction on the order.
The measures must satisfy a combined Wasserstein/Total Variation-continuity assumption, which is one of the weakest conditions used in the context of viscosity approach for this type of integro-differential PDEs.  The proof relies on a regularizing effect due to the gradient growth. We present several examples of applications to PDEs with different types of nonlocal operators (measures with density, operators of variable order, Lévy-Itô operators).

 


\section{Introduction}\label{sec:intro}

In this paper, we are interested in the well-posedness of degenerate elliptic nonlocal Hamilton-Jacobi equations with superlinear gradient growth. The model equation takes the form
\begin{equation}\label{eq:nHJ}
	\lambda u - \mathcal{I}_xu(x) + H(x, Du) = 0 \quad \text{in} \quad \R^N,
\end{equation}
where $\lambda > 0$ is a constant, $H:\R^N\times \R^N\to\R$ is a continuous Hamiltonian, superlinear and coercive with respect to the gradient variable, and $\I_x$ is an integro-differential operator of Lévy type, that is, for $u: \R^N \to \R$ measurable, and $\xi, x \in \R^N$, it writes
\begin{equation}\label{eq:Levy}
	\I_{\xi}u(x) = \int_{\R^N} \left( u(x + z) - u(x) - \1_B(z) Du(x)\cdot z \right) \nu_{\xi}(dz),
\end{equation}
whenever the integral makes sense.  Throughout the paper we assume $\nu_\xi$ is a measure in $\R^N$ with $\nu_\xi(\{ 0 \}) = 0$, and such that the family $\left( \nu_\xi \right)_{\xi\in\R^N}$ satisfies the uniform Lévy condition
$$\sup_{\xi\in \R^N} \int_{\R^N} \min(1,|z|^2) \nu_\xi(dz) < \infty.$$
Under this assumption, $\I_\xi u(x)$ makes sense if $u$ is bounded and sufficiently regular at $x$, say $C^2$. 

The types of Lévy measures we consider include, but are not restricted to, Carathéodory density measures, i.e. measures for which there exists $K:\R^N\times\R^N \to\R_+$  such that
 \begin{equation}\label{eq:density-measure} 
 	\nu_\xi(dz) = K(\xi,z) dz.
\end{equation}	
Within this case, of particular interest are kernels comparable to the one of the fractional Laplacian of order $\sigma \in (0,2)$, in the sense that there exist $0\leq c \leq C$, such that, 
\begin{equation}\label{eq:unif-elliptic}
	 \frac{c}{|z|^{N+\sigma}}  \leq K(\xi,z)  \leq \frac{C}{|z|^{N+\sigma}} \quad \mbox{for all} \ \xi \in \R^N, \ z \neq 0.
\end{equation}

We prove existence and uniqueness of bounded, continuous solutions of \eqref{eq:nHJ} within the framework of viscosity solutions' theory. Our main result, Theorem~\ref{thm.comparison}, establishes a comparison principle between bounded, discontinuous viscosity sub- and supersolutions of \eqref{eq:nHJ}. This, coupled with Perron's method, further allows us to obtain well-posedness for problem \eqref{eq:nHJ}, which is the content of Theorem~\ref{thm.well-posed}.
\smallskip

The main novelty of the results presented in this article are related to the generality of the Lévy measures involved in the equation. At this respect, we shall mention that comparison results for general partial integro-differential equations (PIDEs) involving $\xi-$dependent operators of the form \eqref{eq:Levy} remain one the most important open questions in the field, even though progress has been made in this direction. Several results have been obtained, according to the {\em nature of the integral operator} - general Lévy or Lévy-Itô, or to {\em the order of the operator} - often reflected in the singularity at the origin of the L\'evy measure against which they are integrated. 
\smallskip

In the first case, we recall that a Lévy-Itô operator takes the form
\[ \mathcal J_\xi u(x) = \int_{\R^N} \left(u(x + j(x,z)) - u(x) - \mathbf{1}_{B}(z) Du(x) \cdot j(x,z)\right) \nu(dz), \]
where, for each $\xi \in \R^N$, $j(\xi, \cdot) : \R^N \to \R^N$ is a measurable function known as the \textsl{jump function}, and $\nu$ is a Lévy measure. Under adequate assumptions on $j$, the Lévy-Itô case is relatively well understood, either by means of optimal stochastic control or by direct viscosity methods. Some of the first existence and comparison results were established by Soner in \cite{S86b}. The viscosity theory for general PIDEs has been revisited by Barles and Imbert in \cite{BI08}. Therein, the authors provide a general variant of Jensen Ishii's Lemma for a large class of PIDEs. However,  comparison principles are only obtained for {\em nonlocal operators in Lévy-Itô form}. Their results generalize the ones obtained by Imbert in \cite{I05}  for first-order Hamilton Jacobi equations, and by Jakobsen and Karlsen \cite{JK05, JK06} and Arisawa in \cite{Ari06} for second order PIDEs. Strong comparison results were later given by Ciomaga in \cite{C12}. Barles, Ley and Topp extended  in\cite{BT16} the comparison results for PIDEs where the nonlocal diffusion is coupled with a dominant, coercive Hamiltonian. Chasseigne and Jakobsen in~\cite{CJ17} provided comparison results for quasilinear PIDEs with Lévy-Itô structure.
\smallskip

The general case of $\xi$-dependent Lévy operators is by far less understood. General existence and comparison results for semi-continuous (and yet unbounded) viscosity solutions were found by Alvarez and Tourin in~\cite{AT96} when the kernels involved are finite. In the case of space-time, bounded measures, results were provided by Amadori \cite{A03} and Alibaud \cite{A07}, for continuous viscosity solutions. For PIDEs involving \emph{first order operators}, meaning that  
\begin{equation}\label{order<1}
\sup_{\xi\in \R^N} \int_{\R^N} \min(1,|z|) \nu_\xi(dz) < \infty,
\end{equation}
results were established by Soner in \cite{S88}, followed by the works of Sayah in \cite{S91a}. 
\smallskip

In the attempt of dealing with general nonlocal operators {\em up to the second order}, most of the times comparison results rely on extra assumptions. For instance, for equations ``in divergence form", uniqueness results for weak (or distributional) solutions can be obtained through Maximum Principles, see for instance~\cite{DT23}. In terms of the nonlocal operator, some symmetry assumption on the kernel is required. 
Coming back to the ``non-divergence" setting, Mou and {\'S}wi{\c{e}}ch~\cite{MS15} proved comparison principle for a class of nonlinear PIDEs including Bellman-Isaacs equations when {\em a priori regularity} of the involved sub and/or supersolution is known. This is a rather standard principle in the viscosity formulation, in the sense that the more regularity of the functions to compare we have, the better the control of the terms arising from the penalization introduced in the doubling variables method is.

\smallskip

An innovative step forward can be found in the paper of Guillen, Mou and {\'S}wi{\c{e}}ch~\cite{GMS19}, where new comparison principles among continuous solutions of PIDEs are obtained, for a wide class of nonlocal operators, by means of Optimal Transport techniques. In particular, they impose a Lipschitz continuity condition with respect to the $L^p$ transport metric, with the exponent $p\in[1,2]$ related to the order of the singularity at the origin. Though this approach still validates comparison principles for nonlocal equations of lower order (i.e.,~\eqref{order<1} holds), or when one of the functions to compare has higher a priori regularity, the methods introduced in~\cite{GMS19} allow to consider far more general family of measures which do not fit into the assumptions previously addressed.
\smallskip

The current work continues the research in the direction introduced in~\cite{GMS19}. Our main contribution is a comparison principle between discontinuous viscosity sub- and supersolutions, for {\em general $\xi$- dependent Lévy operators, without any restriction on the order}. The interest of the result is twofold: (i) there is no restriction on the order of the operator, and (ii) we do not need to assume any a priori regularity neither of the subsolution nor of the supersolution, aside from the upper and the lower semicontinuity. This last point is of particular interest since this implies that no uniform ellipticity (in the sense of Caffarelli and Silvestre~\cite{CS09}, that is $c > 0$ in~\eqref{eq:unif-elliptic}), or even weak ellipticity (as in \cite{BCCI12}) of the nonlocal operator is required. In fact, the nonlocal diffusion could be very degenerate in nature, and could be of mixed type, (that is, it can be combined with a local diffusion), or with variable order, see Sections \S~\ref{sec:extensions} and \S~\ref{sec:examples} for a more accurate exposition of these settings.
\smallskip

At this point, we mention that such level of generality for the nonlocal operator is possible by the presence of $H$ in the equation, and we are still far from the most general comparison result among discontinuous viscosity sub and supersolutions. In fact, we get the result when we combine the nonlocal diffusion with a superlinear coercive Hamiltonian in the gradient variable. To fix ideas, the reader could bear in mind the superlinear eikonal model, where the Hamiltonian takes the form
\begin{equation}\label{eq:eik-H}
H(x,p)=b(x)|Du|^m - f(x)
\end{equation}
where $b,f: \R^N \to \R$ are bounded, continuous functions, with $b(\cdot) \geq b_0> 0$, and $m>1$. 
\smallskip

Roughly speaking, the degenerate elliptic nature of the nonlocality is controlled by the action of the Hamiltonian in the doubling variables procedure: if $u, v$ are respectively a viscosity sub- and super-solution to the problem, we double variables and penalize as usual, arguing over the function
$$
(x,y) \mapsto u(x) - v(y) - \epsilon^{-2}|x - y|^2, \quad x, y \in \R^N.
$$
Up to a localization, we have the existence of a maximum point $(\bar x, \bar y)$ of the above function, and we employ the viscosity inequalities for $u$ and $v$ at $\bar x, \bar y$ respectively. Then, under the assumptions of the data, the leading contribution of the Hamiltonian allows us to prove that the ``viscosity gradient"
$$
\frac{\bar x - \bar y}{\epsilon^2}
$$
is bounded, uniformly in $\epsilon$ (see Lemma~\ref{lem.p-bound}). This is exactly the type of conclusion we get, for instance, if one of the functions to compare is smooth. 
\smallskip

Of course, the continuity of the coefficients with respect to the state variable is a structural assumption of the method. As such, the hypotheses on the map $\xi \mapsto \nu_\xi$ are adapted to the singularity of the measures around the origin. Namely, we assume continuity in the total variation distance away from the origin, i.e. for $0<r<R$ there exists a modulus of continuity $\omega_{r,R}$ such that
\[ 
\vert \nu_x - \nu_y\vert(B_R\setminus B_r) \leq \omega_{R,r}(\vert x- y\vert), \quad \mbox{for all} \ x,y\in\R^N.
\]
On the other hand, we assume a 1/2-H\"older continuity on the Wasserstein distance ($L^2$ transport metric), localized around the origin, i.e. for $0<r<1$ we have
\[ 
W_2(\nu_x,\nu_y)(B_r) \leq o_r(1)\vert x- y\vert^{1/2}, \quad \mbox{for all} \  x,y\in\R^N.
\]
We refer the reader to Section \S\ref{subsec:assumptions} for precise definitions and hypotheses. These are sufficient conditions to conclude the result, as they lead to a ``regularizing effect" of the Hamiltonian into the equation, even though we neither use nor we prove any regularity property on $u$ or $v$. We point out that, one could invoke the H\"older estimates for {\em subsolutions} to~\eqref{eq:nHJ} when the gradient term dominates the diffusion. In terms of~\eqref{eq:unif-elliptic} and~\eqref{eq:eik-H}, this means that $m > \sigma$, see~\cite{BKLT15, BT16b}. On the other hand, if the nonlocality is uniformly elliptic and the diffusion rules the equation (that is, $c > 0$ in~\eqref{eq:unif-elliptic} and $m \leq \sigma$), it is possible to employ elliptic estimates to get H\"older estimates for continuous solutions, see~\cite{BCI11}. Nonetheless, our results cover the intriguing scenario where neither the ellipticity nor the coercivity rules the equation, namely, when $1 < m \leq \sigma$, and $\I_x$ is degenerate elliptic.
\smallskip

We finish this introduction by mentioning that these types of equations occur in the study of stochastic optimal control problems with jump-diffusion processes. In this setting, $\I_\xi$ is the infinitesimal generator of a jump process either in Lévy-Itô form, or in Courrège form, see \cite{App09}. In the local case (namely, when the diffusion is governed by a continuous process like the Brownian motion), a superlinear Hamiltonian with the form~\eqref{eq:eik-H} is related to trajectories controlled with unbounded drift, and the value function of the control problem is the viscosity solution of the associated Hamilton Jacobi equation, see~\cite{LL89, BF92}. We expect to have similar verification results in the nonlocal context, see for instance~\cite{OS05}. Since we do not assume convexity on $H$, this type of result has potential application to differential games.

\medskip

{The paper is organized as follows.} We introduce the notations and definitions, make precise the set of assumptions and state the main results in Section \S\ref{sec:assumptions}. In Section \S\ref{sec:main-proof} we prove the main comparison result. We then provide a couple of extensions of the Comparison Principle to the second order and time dependent case in Section \S\ref{sec:extensions}. Finally, we present a set of examples that illustrate the scope of the assumptions, which include interesting measures without density or measures with intricate elliptic properties in Section \S\ref{sec:examples}.

\section{Assumptions and main results}\label{sec:assumptions}

\subsection{Notations} 

The scalar product of $x, y \in \R^N$ is denoted by $x\cdot y$ and the Euclidean norm of $x$ by $|x|$. For $x \in \R^N$ and $r > 0$,  $B_r(x)$ is the open ball in $\R^N$, centered at $x$, of radius $r>0$. 
We write $B_r$ whenever the ball is centered at the origin and has radius $r$.  We denote by $B$ the unit ball, and by $B^* = B\setminus \{0\}$ the punctured ball.  
For $A$ a subset of  $\;\R^N$, let $USC(A)$, $LSC(A)$ be the spaces of upper semi-continuous functions on $A$ and lower semi-continuous functions on $A$, respectively.  Let $C(A)$ be the space of continuous functions on $A$,  and, for $k\ge 1$, let $C^k(A)$ be the class of $k-$continuously differentiable functions on $A$. 
The gradient of a differentiable function $\phi:\R^N\to\R$ is denoted by $D\phi$, and the Hessian matrix of a twice differentiable function $\phi$ is denoted by $D^2\phi$.
We write $\mathbb S^N$ for the space of $N \times N$ symmetric matrices with real entries and  denote $I$ the $N \times N$ identity matrix.

We introduce the following notations in order to evaluate the integral terms.  Let $A\subset \R^N$ be a measurable set.  
For $\xi, x,p\in\R^N$ and $\phi \in L^\infty(\R^N)$, we define
\begin{equation*}
\mathcal I_\xi[A](x,p,\phi) = \int_A \big( \phi(x + z) - \phi(x) - \1_B(z) p \cdot z \big) \nu_\xi(dz).
\end{equation*}
If $\phi \in C^2(A\cap B_r(x))\cap L^\infty(\R^N)$, and $p= D\phi(x)$, then we write
\begin{equation*}
\mathcal I_\xi[A](x,\phi) = \int_A \big( \phi(x + z) - \phi(x) - \1_B(z) D\phi(x) \cdot z \big) \nu_\xi(dz).
\end{equation*}
In the case when $\xi = x$, we drop the $\xi$ dependence and simply write 
$\I[A](x, \phi) := \I_x[A](x,\phi)$ and
$\I[A](x, p,\phi) := \I_x[A](x,p,\phi)$ respectively.

Finally, $\omega : \R^+\to \R^+$ is  a modulus of continuity if it satisfies $\omega(s)\to 0$ as $s\to 0$.
We also employ the usual notation $o_\eps(1)$ to express a quantity $q(\eps) \to 0$ as $\eps\searrow 0$. For given parameters $\gamma_1, \cdots, \gamma_k$, and $\eps$, we write 
$\displaystyle o_\eps^{\gamma_1, \cdots, \gamma_k}(1)$ to express a quantity  
$ q(\eps; {\gamma_1, \cdots, \gamma_k}) \searrow 0 \; \text{ as }\;  \eps \searrow 0,
\; \text{ for fixed } \; \gamma_1, \cdots, \gamma_k.$

\subsection{Distances between Lévy measures}

We use in this paper two notions of metrics between Lévy measures: the Wasserstein distance and the total variation.

Following the notion of boundary Wasserstein metric between finite measures on bounded sets, introduced by Figalli and Gigli in ~\cite{FG10}, and later developed for L\'evy measures by Guillen, Mou and \'Swie\c ch in~\cite{GMS19}, we consider the following notion of Wasserstein distance. For a given measurable set $A\subset\R^N$, let $\mathcal M(A)$ be the set of L\'evy measures on $A$ -- possibly singular at the origin, that is
\begin{equation*} 
\mathcal M (A) := \left\{ \nu \text{ positive Borel measure on } A 
	\; :\; \nu(\{0\}) = 0, \; \int_{A\setminus \{0\}} \min(1,|z|^2) \nu (dz) < \infty \right\}. 
\end{equation*}	

Let $A$ be a subset of the unit ball $B$ such that $0\in A$. Given $\nu_1, \nu_2 \in \mathcal M(A)$, we define the set of admissible couplings for $\nu_1, \nu_2$ as the set of positive measures on $A \times A$ whose marginals on $A\setminus \{0\}$ are $\nu_1$ and $\nu_2$ respectively, that is
\begin{eqnarray*} 
{\text{Adm}_A}(\nu_1,\nu_2) & : = & \left\{ \gamma \text{ positive measure on } A \times A
	:  \gamma (\{(0,0)\}) = 0, \;  \pi^i_\# \gamma\mid_{A \setminus \{0\}} = \nu_i, 
	\text{ for } i=1,2 \right\},
\end{eqnarray*}	
where $\pi^i : \R^{2N}\to \R^N$, with $\pi^i(x_1,x_2) = x_i$, is the canonical projection in the $i$-th coordinate. 

\begin{definition}[Wasserstein distance]\label{def.Wass}
Let $p \in [1, 2]$ and $\nu_1, \nu_2 \in \mathcal M(B)$. For any $\,0<r<1$, we define 	
the \emph{Wasserstein distance between $\nu_1$ and $\nu_2$, restricted on $B_r$,} as
\[ W_p(\nu_1,\nu_2)(B_r):= 
			\left( \inf_{\gamma\in {\mathrm{Adm}_{B_r}(\nu_1^{r},\nu_2^{r})}} 
			\int \limits_{B_r\times B_r} |z_1 - z_2|^p d\gamma(z_1,z_2) \right)^{1/p},\]
where $\nu_i^r := \nu_i\mid_{B_r}$ is the restriction of the measure $\nu_i$ to $B_r$.
\end{definition}

\begin{remark}\em
Similarly to ~\cite{GMS19}, the admissible set $\mathrm{Adm}_A(\nu_1, \nu_2)$ considers measures defined on $A \times A$, despite the fact that marginals are restricted to $A \setminus \{ 0 \}$. This is due to the fact that Lévy measures are singular at the origin, therefore the origin can be used as an infinite reservoir: we can take mass and redistribute within $B$ or, conversely, take infinitely large mass from $B$ and send it to the origin, provided we pay an adequate transportation cost. This plays a substantial role in the definition. On the contrary, L\'evy measures being finite elsewhere, the boundary of the unit ball can be treated as having zero mass. Therefore, working with measures defined on the open ball $B$ or the closed ball $\ovl B$ shall not matter, but what matters is that mass is concentrated at the origin.  Additionally, as pointed out in~\cite{FG10}, the form of the cost function $c(z_1, z_2) = |z_1 - z_2|^p$ in the definition of Wasserstein distance allows us to assume without loss of generality that $\gamma(\{ (0, 0) \}) = 0$ for every $\gamma$ admissible. This justifies the choice in the definition of admissible couplings. %
\end{remark}

We also employ the total variation between two Lévy measures, that we recall here for simplicity.
\begin{definition}[Total variation]\label{def.TV}
Let $\nu_1, \nu_2 \in \mathcal M(\R^N)$. We call \emph{the total variation between $\nu_1, \nu_2$} the positive measure
$$ |\nu_1 - \nu_2| = (\nu_1 - \nu_2)^+ + (\nu_1 - \nu_2)^-, $$
where $(\nu_1 - \nu_2)^\pm$ are given by the Jordan decomposition $\nu_1 - \nu_2 = (\nu_1 - \nu_2)^+ - (\nu_1 - \nu_2)^-$ on the signed measure $(\nu_1 - \nu_2)$, that is,
for any measurable set $A\subset \R^N\setminus \{0\}$,
\begin{eqnarray*}
  (\nu_1-\nu_2)^+(A) & := & \;\; \sup \big\{ \nu_1(E) - \nu_2(E) : 
        E \subseteq A, \ E \ \mbox{ measurable } \big\} \\
  (\nu_1-\nu_2)^-(A) & := & - \inf \big\{ \nu_1(E) - \nu_2(E) : 
        E \subseteq A, \ E \ \mbox{ measurable } \big\}.
\end{eqnarray*}
\end{definition}

We list below a series of basic properties of the Wasserstein distance between L\'evy measures, used throughout the paper.  We concentrate on Lévy measures supported on the unit ball $B$, but all the results can be readily adapted to other local domains of $\R^N$.

\begin{lemma}\label{lem.supp-gamma}
Let $\nu_1, \nu_2 \in\mathcal M(B)$ such that ${\rm supp} (\nu_1) = A_1, {\rm supp} (\nu_2) = A_2$. Then, for each admissible coupling $\gamma \in \text{Adm}_B(\nu_1, \nu_2)$ we have that 
$ \displaystyle \; {\rm supp}(\gamma)\subset(A_1 \times A_2)\cup(A_1\times\{0\})\cup(\{0\}\times A_2).$
\end{lemma}

\begin{proof} 
Since $\gamma(\{(0,0)\})=0$, without any loss of generality, we may assume $0\in A_1 \cap A_2$. We show that 
 ${\rm supp}(\gamma)\subset A_1 \times A_2$. We argue by contradiction and assume there exists a set $E\subset (B\times B) \setminus (A_1\times A_2)$ such that $\gamma(E)>0$. 
 We decompose the set into the disjoint union $E = E_1 \cup E_2\cup F$, with 
 \[ E_1 = E\cap A_1 \times B, \quad E_2 = E\cap B \times A_2, \quad 
 F \subset A_1^c\times A_2^c.\]
 Since $\gamma$ is an admissible coupling with marginals $\nu^i$, $\pi^1(E_2) \subset B^*\setminus A_1$
 and $(\pi^1)^{-1}(\pi^1 (E_2))\supset E_2$, we have
\begin{eqnarray*}
  0 = \nu_1(\pi^1(E_2)) & = & \pi^1_\# \gamma  \mid_{B^*} (\pi^1(E_2)) =   \pi^1_\# \gamma  (\pi^1(E_2)\cap B^*) \\
  & = & \gamma((\pi^1)^{-1}(\pi^1(E_2))\cap (B^*\times B)) = \gamma (E_2\cap (B^*\times B))   \geq \gamma(E_2).
 \end{eqnarray*}
Similarly, we can check that $\gamma(E_1)= 0$ and $\gamma(F) = 0$. Summing up, we reach a contradiction with the assumption $\gamma(E)>0$.
\end{proof}

An important property employed in our arguments is related to the integration of functions of one variable against admissible couplings between two L\'evy measures. 
\begin{prop}\label{prop.split}
Let $\nu_1, \nu_2\in \mathcal M(B)$ and $\gamma \in \mathrm{Adm}_B(\nu_1, \nu_2)$ an admissible coupling.  
If $f$ is a measurable function with support away from the origin (there exists $\delta > 0$ so that ${\rm supp}(f) \subset B \setminus B_\delta$), then 
\begin{eqnarray*}
	\int_{B \times B} f(z_i) d\gamma(z_1, z_2) & = & \int_B f(z)\nu_i(dz) \quad i=1,2.
\end{eqnarray*}
\end{prop}

\begin{proof}
For simple functions $f = \sum_{k=1}^n c_k \chi_{A_k}$ with $c_k \geq 0$ and $A_k \subset \R^N$ measurable, we can always assume that $A_k \cap B_\delta = \emptyset$. We have
\begin{eqnarray*}
\int_{B \times B} f(z_1) d\gamma(z_1, z_2) 
	& = &  \sum_{k=1}^n c_k \int_{B \times B} \chi_{A_k}(z_1) d\gamma(z_1, z_2)
	   	= \sum_{k=1}^n c_k \gamma(A_k \times B).
\end{eqnarray*}
Taking into account that $\gamma$ has marginals $\nu_i$, i=1,2, and $A_k \cap B_\delta = \emptyset$, we get
$\gamma(A_k \times B)=\gamma((A_k\cap B^*) \times B)=  \pi^1_\# \gamma \mid_{B^*} (A_k)=\nu_1(A_k).$ It follows that
\begin{eqnarray*}
\int_{B \times B} f(z_1) d\gamma(z_1, z_2) 
	& = &   \sum_{k=1}^n c_k \nu_1(A_k) 
		=  \sum_{k=1}^n c_k \int_B  \chi_{A_k}(z_1)\nu_1(dz_1) 
		= \int_B  f(z_1)\nu_1(dz_1).
\end{eqnarray*}  

For general nonnegative measurable functions, it is possible to conclude by the previous argument and the Monotone Convergence Theorem. From here, the result follows for a general measurable function with support away the origin, by writing $f = f^+ - f^-$ and using the previous results.
\end{proof}

An immediate corollary of the above lemma is the following result. 
\begin{prop}\label{prop.split.f+g}
Let $\nu_1, \nu_2\in \mathcal M(B)$ and $\gamma \in \mathrm{Adm}_B(\nu_1, \nu_2)$ an admissible coupling. 
If $f, g: B \rightarrow \R$ are measurable functions with support away the origin, then  
\begin{equation*}
\int_{B \times B}[f(z_1) + g(z_2)] d\gamma(z_1, z_2) = \int_{B} f(z) \nu_1(dz) + \int_{B} g(z) \nu_2(dz).
\end{equation*}
\end{prop}

If instead, we are integrating functions supported in a ball --such as a quadratics, the equality fails, but we still have a control of the integral against the admissible coupling, by the integral of one variable only. More precisely, the following holds.
\begin{prop}\label{lem.Gigli}
Let $\nu_1, \nu_2 \in \mathcal M(B)$ and  $\gamma \in \mathrm{Adm}_B(\nu_1, \nu_2)$ an admissible coupling. Then
\[  \int_{B \times B} |z_1|^2 d\gamma(z_1, z_2) \leq 4 \int_{B} |z|^2 \nu_1(dz). \] 
\end{prop}
\begin{proof}
Decompose the integration domain as 
$ \displaystyle B \times B = \bigcup \limits_{k=0}^\infty A_k \times B \quad 
\text{with }	A_k = B_{2^{-k}} \setminus B_{2^{-(k + 1)}}. $
Then, taking into account that $\gamma$ is an admissible coupling and
 $\gamma(A_k\times B)=\gamma((A_k\cap B^*)\times B)=\nu_1(A_k)$, it follows that
\begin{eqnarray*}
\int_{B \times B} |z_1|^2 d\gamma(z_1, z_2) 
	&  = & \sum_{k=0}^\infty \int_{A_k \times B} |z_1|^2 d\gamma(z_1, z_2)
		\leq \sum_{k=0}^\infty 2^{-2k} \int_{A_k \times B} d\gamma(z_1, z_2) \\
	& =  &  \sum_{k=0}^\infty 2^{-2k} \int_{A_k}  \nu_1(dz_1)
		 \leq  4\sum_{k=0}^\infty \int_{A_k} |z_1|^2 \nu_1(dz_1).
\end{eqnarray*}
\end{proof}

Finally, the Wasserstein distance is always controlled by the second moment of the total variation. 
\begin{prop}\label{prop.Wass-TV}
Let $\nu_1, \nu_2 \in\mathcal M(B)$. Then
\begin{equation}
	W_2^2(\nu_1, \nu_2)(B) \leq \int_{B} |z|^2 |\nu_1 -\nu_2|(dz).
\end{equation}
\end{prop}
\begin{proof}
We aim at constructing a particular admissible coupling between the measures $\nu_1$ and $\nu_2$ that yields the above bound. To this end, we keep the common mass between the two measures via the identity map $Id$ and transport mass from and onto the origin, via the projection function $P: B \to B$ given by $P(x) = 0$, for all $x\in B$. Given the Lévy measures $\nu_1, \nu_2 \in\mathcal M(B)$, we consider the 
the minimal measure $\nu_1 \wedge \nu_2$, defined by
$  \nu_1 \wedge \nu_2 = \nu_1 - (\nu_1-\nu_2)^+. $
Consider then the coupling  transport plan $\gamma : B\times B \to\R_+$ given by 
\[  \gamma = (Id \times Id)_\# (\nu_1 \wedge \nu_2) 
		 + (Id \times P)_\# (\nu_1-\nu_2)^+ 
		 + (P \times Id)_\# (\nu_1-\nu_2)^-,\]
which is admissible for the pair $(\nu_1,\nu_2)$.  
Indeed, for each  measurable subset $A \subseteq B^*$, we have
\begin{eqnarray*}
\pi^1_\# \gamma   \mid_{B^*} (A) 
	& = & \gamma((\pi^1)^{-1}(A) \cap (B^*\times B))  = \gamma(A\times B) 
   	 =  (\nu_1 \wedge \nu_2)((Id \times Id)^{-1}(A \times B)) \\
   	&&	+ (\nu_1-\nu_2)^+((Id \times P)^{-1}(A \times B))
	    + (\nu_1-\nu_2)^-((P \times Id)^{-1}(A \times B))\\
  	        &  = &  (\nu_1 \wedge \nu_2)(A) + (\nu_1-\nu_2)^+(A) + (\nu_1-\nu_2)^-(\emptyset)
                =  \nu_1(A).
\end{eqnarray*}
Similarly, we get $\pi^2_\#  \gamma\mid_{B^*}(A) = \nu_2(A)$  and $\gamma(\{(0,0)\})=0$. 
Thus, the transport plan $\gamma$ above is admissible. 
Computing with changes of variables, we obtain
\begin{align*}
\int_{B \times B} |z_1 - z_2|^2 d\gamma(z_1, z_2) 
	& =  \int_{B} |z|^2 (\nu_1-\nu_2)^+(dz) + \int_{B} |z|^2 (\nu_1-\nu_2)^-(dz) 
	 =  \int_{B} |z|^2 |\nu_1-\nu_2|(dz).
\end{align*}
Taking the infimum over all admissible couplings $\gamma$,  the conclusion follows.
\end{proof}

\subsection{Assumptions}\label{subsec:assumptions}

We now state the assumptions  on the family of Lévy measures and on the Hamiltonian.
We consider  a family of Lévy measures $\big(\nu_\xi\big)_{\xi\in\R^N}\subset\mathcal M(\R^N)$,  satisfying the following conditions:
\begin{itemize}
\item[(M1)] There exist a constant $C_\nu>0$ such that 
		\[ \sup_{\xi \in \R^N} \int_{\R^N} \min(1,|z|^2) \nu_{\xi}(dz) \leq C_\nu. \]
\item[(M2)] There exists $R_0>1$ such that, for all $R\ge R_0$, 
		\[ \sup_{\xi \in \R^N} \int_{B_R^c} \nu_{\xi}(dz) \leq o_{1/R}(1). \]	
\item[(M3)] For all $0<r< R$, there exists a modulus of continuity 
		$\omega_{R,r}: \R_+ \to \R_+$ so that, for all $x,y\in\R^N$, 
		\[|\nu_x - \nu_y|\left(B_R\setminus B_r\right)\le\omega_{R,r}(|x-y|).\]	
\item[(M4)] There exists $0<r_0<1$ such that, for all $0<r<r_0$ and for all $x,y\in\R^N$, 
        \[ W_2(\nu_x, \nu_y)(B_r) =o_r(1) |x - y|^{1/2}. \]
\end{itemize}

Assumption (M1) is a uniform Lévy condition, whereas (M2) is a uniform decay at infinity, which is a piece of interest when we deal with problems in $\R^N$.
The next two assumptions are the key points in this research and they are assumed to hold \emph{without any restriction on the order of the nonlocal operator.} Due to the singularity at the origin of the Lévy measure, we are often led to split the nonlocal operator into sets around the origin and away from the origin.
Assumption (M3) deals with the bounded part of the Lévy measures and it involves the total variation. Assumption (M4) deals with the singular part of the Lévy measures and it involves the Wasserstein distance. A common assumption would have been a unified continuity condition, expressed in terms of the total variation only:
\begin{itemize}
\item[(M)] For all $r>0$, and for all $x,y\in\R^N$, 
		\[\int_{B_r} \min(1, |z|^2) | \nu_{x}- \nu_{y}| (dz) \leq o_r(1)|x-y|.\]	
\end{itemize}
 The continuity condition (M4) with respect to the data has already appeared in the study of the regularity of solutions, see for example \cite{BCI11, BCCI12, BLT17}. Recently, Guillen, Mou and {S}wie\c ch showed in \cite{GMS19} that, within the study of comparison principles for integro-differential equation, by means of doubling of variables technique, the quadratic optimal transport distance naturally arises as a nonlocal Jensen-Ishii variant. In addition, the Wasserstein distance is to be considered only on small balls around the origin. This explains the coupling (M3)-(M4).

It is easy to see that a family of measures satisfying assumption (M) will satisfy (M3) and (M4).

\begin{prop}\label{prop.weakTV}
Let $\big(\nu_\xi\big)_{\xi\in\R^N}\subset\mathcal M(\R^N)$ be a family of Lévy measures which satisfies $(M)$. 
Then  $(M3)$ and $(M4)$ hold.
\end{prop}
\begin{proof}
Writing (M) for $r=1$, we have for all $0<r\leq 1$,
\begin{eqnarray*}
  && \int_{B} |z|^2 | \nu_{x}- \nu_{y}| (dz)
  =  \int_{B_r} |z|^2 | \nu_{x}- \nu_{y}| (dz)+ \int_{B \setminus B_r} |z|^2 | \nu_{x}- \nu_{y}| (dz)
  \leq  o_1(1)|x-y|.
\end{eqnarray*}
Hence
\begin{eqnarray*}
  &&   \int_{B \setminus B_r}  | \nu_{x}- \nu_{y}| (dz)
  \leq   \frac{1}{r^2} \int_{B \setminus B_r}  |z|^2 | \nu_{x}- \nu_{y}| (dz)
  \leq \frac{o_1(1)}{r^2} |x-y|. 
\end{eqnarray*}
Therefore, for $0<r\leq 1\leq R$,
\begin{eqnarray*}
&&  | \nu_{x}- \nu_{y}|(B_R\setminus B_r)  =
  \int_{B \setminus B_r}  | \nu_{x}- \nu_{y}| (dz) +  \int_{B_R \setminus B}  | \nu_{x}- \nu_{y}| (dz)
  \leq \left( o_R(1)+  \frac{o_1(1)}{r^2} \right) |x-y|,
\end{eqnarray*}
where we have used (M) for $R>1$ to estimate the total variation on $B_R\setminus B$. This gives (M3).
Assumption (M4) follows from Proposition \ref{prop.Wass-TV} together with (M) for $r<1$.
\end{proof}

\begin{remark}{Assumptions (M3)-(M4) are, at least localy,  weaker than assumption (M). To  illustrate  this fact, we construct a family of measures for which (M3)-(M4) hold for $|x-y|$ small, but (M) does not. 
\rm  Let 
\[  \nu(dz) = \mathbf{1}_Q(z) |z|^{-(2 + \sigma)} dz, \quad \mbox{with} \ Q = \{ (x,y) \in \R^2 : x,y > 0 \} \subset \R^2. \]
For any point on the unit sphere $\mathrm S$, $x =   e^{i\theta_x} $ with $\theta_x \in(-\pi,\pi] $ consider the rotation $R_x:\R^2\to\R^2$ given by $R_x(z) = e^{i\sqrt{  |\theta_x| }}z$. 
Define the family of Lévy measures indexed with respect to points on the unit sphere, $\left( \nu_x\right)_{x\in S}$ given by $\nu_x = \left(R_x\right)_\# \nu$. Note that 
\[  \nu_x(dz) = \mathbf{1}_{Q_x}(z) |z|^{-(2 + \sigma)} dz, \quad \mbox{with} \ Q_x = \{ R_x(z) : z\in Q\} \subset \R^2. \]

In order to check $(M4)$ and evaluate the Wasserstein distance, we construct
$\gamma_{x,y}:= (R_x \times R_y)_\# \nu^r$ which is an admissible coupling in $\textrm{Adm}_{B_r}(\nu_x^r,\nu_y^r).$
Indeed, for $A \subseteq B_r \setminus \{ 0 \}$, we see that
\[  \gamma_{x,y}(B_r \times A)  = \nu^r ((R_x \times R_y)^{-1}(B_r \times A)) = \nu^r(R_x^{-1}(B_r)\cap R_y^{-1} A) 
= \nu^r(R_y^{-1} A) = (R_y)_\# \nu(A) =  \nu_{y}^r(A) \]
and similarly $\gamma_{x,y}(A\times B_r) =   \nu_{x}^r(A)$. Hence, we have, for $|x-y|$ small,
\begin{eqnarray*}
W_2^2(\nu_x, \nu_y)(B_r) &\leq & \int_{B_r \times B_r} |z_1 - z_2|^2 d\gamma_{x,y}(z_1, z_2) 
	=  \int_{B_r} |R_x(z) - R_y(z)|^2 d \nu(z)   \\
&\leq & 2(1 - \cos\left(|\theta_x|^{1/2} - |\theta_y|^{1/2}\right) \int_{Q \cap B_r} |z|^{-\sigma} d z \\
&\lessapprox & C r^{2-\sigma} \left|  |\theta_x|^{1/2} - |\theta_y|^{1/2}\right|^2 
	\leq C r^{2-\sigma} \left|  |\theta_x| - |\theta_y|\right|,
\end{eqnarray*}
where we have used that $1 - \cos(\alpha) \sim \alpha^2/2$ for $\alpha$ small. On the other hand, we have that
\[ | x- y | = \sqrt{2(1 - \cos\left(|\theta_x| - |\theta_y|\right)} \approx \left|  |\theta_x| - |\theta_y|\right|. \]
We conclude the existence of a constant $C > 0$ such that
\[ W_2(\nu_x, \nu_y)(B_r) \leq C r^{(2 - \sigma)/2} |x- y|^{1/2}. \]

We now check the total variation condition (M3). Let $x,y\in S$ and assume, without loss of generality, 
that $\theta_x>\theta_y>0$. Then, it holds by similar arguments, that 
\[ |\nu_{x} - \nu_y |(B_R \setminus B_r) 
 	= \int_{(B_R \setminus B_r) \cap (Q_x \Delta Q_y)} |z|^{-(2 + \sigma)} dz 
 	=  2  C \int_{r}^{R} \rho d\rho \int_{\sqrt\theta_y}^{\sqrt \theta_x} d\theta
	 \lessapprox C \left(r^{-\sigma} - R^{-\sigma} \right) |x - y|^{1/2}.\] 
On the other hand, we see that
\[ \int_{B_r} |z|^2 |\nu_x - \nu_y|(dz) 
	= \int_{B_r\cap Q_x \Delta Q_y} |z|^{-\sigma} dz 
 	= \frac{r^{2 - \sigma}}{2-\sigma}  \left|\sqrt{\theta_x} - \sqrt\theta_y \right| 
	\gtrapprox C r^{2 - \sigma} |x - y|^{1/2},
\]
for some $C > 0$ small enough, but independent of $x$ or $y$. Thus, we do not have the Lipschitz condition (with respect to $|x - y|$) on the second momentum of the total variation.}

\end{remark}
\medskip

We finally make precise the assumptions on the Hamiltonian. We emphasize  that the key assumption is a superlinear growth in the gradient variable - which will dominate all of the other terms in the equation. 
Let $H:\R^N\times \R ^N\to \R$ satisfy the following:
\begin{itemize}
\item[(H0)] $H \in C(\R^N\times\R^N)$ and it satisfies 
		$\displaystyle \mathop{\rm sup}_{x\in\R^N}|H(x,0)| < +\infty.$
\item[(H1)] There exists $m>1$ and two moduli of continuity 
		$\omega^1_H,\omega^2_H: \R_+\to\R_+$ 
		such that, for all $x,y,p,q\in\R^N$, with $|q|\leq 1$,
		\[ H(y, p + q) - H(x, p) \leq 
			\omega_H^1(|x-y|)(1 + | p |^m) +
			\omega_H^2(| q |)(1 + | p |^{m - 1}). \]
\item[(H2)]  There exists $m>1$, some constants $b_m, b_0 >0$, $r_0>0$
		 and $\mu_0\in (0,1)$ such that for all $\mu\in[\mu_0,1]$
		 and $x,p\in \R^N$ with $|p|\ge r_0$, 
		\[ \mu H(x, \mu^{-1} p) - H(x,p) \ge (1-\mu) (b_m | p |^m - b_0).\]	
\end{itemize}

The first condition is typical in establishing the existence of solutions, by means of Perron's method.
Assumption (H1) is a classical continuity condition on the data in this context, and (H2) expresses
the superlinear coercivity property of $H$. The constant $m\in\R$ in assumptions (H1)-(H2) plays the role of the power of the gradient nonlinearity, and $m>1$ means that we consider superlinear Hamiltonians. More precisely, we have the following estimate, whose proof can be found in~\cite{BCT19}.

\begin{lemma}\label{lem.suplinH}
Let $H$ be a Hamiltonian satisfying assumptions (H0)-(H1). 
Then, for every $R>0$,
\begin{equation*}
	\mathop{\rm sup}_{x\in\R^N, |p|\leq R} |H(x,p)| =: C_R <+\infty.
\end{equation*}
If, in addition, (H2) holds, then there exists a constant $C>0$ (depending on the data from the assumptions) and $r_0>0$ such that, for all $x,p\in\R^N$ and $|p|\geq r_0$,
\begin{equation*}
	H(x,p) \geq C |p|^m - \frac1C |p|.
\end{equation*}
\end{lemma}

\subsection{Main results}

The main contribution of this paper is a comparison result between \emph{discontinuous viscosity sub- and supersolutions}, for elliptic integro-differential equations coupled with a superlinear coercive Hamiltonian of the form~\eqref{eq:nHJ}. We state below the definitions for equation \eqref{eq:nHJ}, used throughout the paper (see \cite{BI08}, for several equivalent definitions). 

\begin{definition}[Viscosity solution]$\;$
 \begin{itemize}
 \item[(i)] A function $u\in USC(\R^N)\cap L^{\infty}(\R^N)$ 
    is {\emph a viscosity subsolution} of \eqref{eq:nHJ} iff, for any test function
    $\phi\in C^2(B)\cap L^\infty(\R^N)$, if $\ovl x$ is a local maximum
    of  $u-\phi$ in $B_\delta$, with $\delta>0$, then
	\[ \lambda u(\ovl x ) - \I[B_\delta](\ovl{x},  \phi) - 
        \I[B_\delta^c](\ovl{x}, D\phi(\ovl x), u) 
	+ H(\ovl{x}, D\phi(\ovl x))  \leq  0. \]
 \item[(i)] A function $u\in LSC(\R^N)\cap L^{\infty}(\R^N)$ 
    is {\emph a viscosity supersolution} of \eqref{eq:nHJ} iff, for any test function  
    $\phi\in C^2(B)\cap L^\infty(\R^N)$, if $\ovl x$ is a local minimum
    of  $u-\phi$ in $B_\delta$, with $\delta>0$, then
	\[ \lambda u(\ovl x ) - \I[B_\delta](\ovl{x},  \phi) - 
        \I[B_\delta^c](\ovl{x}, D\phi(\ovl x), u) 
	+ H(\ovl{x}, D\phi(\ovl x))  \geq  0. \]
 \item[(iii)] A function $u\in C(\R^N)\cap L^{\infty}(\R^N)$ 
    is {\emph a viscosity solution} of \eqref{eq:nHJ} iff it is both a viscosity subsolution and a viscosity supersolution.
 \end{itemize}
\end{definition}

The main result of this paper is the following comparison principle. 

\begin{theorem}[Comparison Principle] \label{thm.comparison}
Let $\lambda >0$, $(\nu_\xi)_{\xi\in\R^N}\subset \mathcal M(\R^N)$ be a family of Lévy measures associated with the nonlocal operators $\left(\mathcal I_\xi(\cdot)\right )_{\xi\in\R^N}$ satisfying assumptions  (M1) - (M4), and let $H$ be a Hamiltonian with superlinear growth, satisfying assumptions (H1)-(H2). 

Then, the comparison principle holds:  if $u\in USC(\R^N)\cap L^\infty(\R^N)$ is a viscosity subsolution of \eqref{eq:nHJ} and  $v\in LSC(\R^N)\cap L^\infty(\R^N)$ is a viscosity supersolution of \eqref{eq:nHJ}, 
then $u\leq v \text{ in } \R^N$.
\end{theorem}

Following Perron's method introduced for viscosity solutions by Ishii in \cite{Ishii87}, one can establish the existence of discontinuous solutions for problem \eqref{eq:nHJ}, where integro-differential terms are linearly coupled with coercive Hamiltonians. However, in order to start Perron's iteration it is necessary to assume the boundedness assumption (H0). This, coupled with the comparison result above leads to the following well-possedness result for problem \eqref{eq:nHJ}.

\begin{theorem}[Well-possedness] \label{thm.well-posed}
Let  $\lambda >0$, $(\nu_\xi)_{\xi\in\R^N} \subset \mathcal M(\R^N)$ be a family of Lévy measures satisfying assumptions (M1) - (M4), and $H$ be a Hamiltonian with superlinear growth, satisfying assumptions (H0)-(H2). Then, there exists a bounded continuous viscosity solution of \eqref{eq:nHJ}.
\end{theorem}

\section{Proof of the comparison principle (Theorem~\ref{thm.comparison}) } \label{sec:main-proof}

The proof is divided into several technical lemmas, which correspond to the main steps. It relies, as usual, on the doubling of variables technique, and fine estimates of the terms coming from the viscosity inequalities. We focus on the most difficult and new part, which concerns fine estimates on the nonlocal terms. 

Let $u\in USC(\R^N)\cap L^\infty(\R^N)$ be a bounded viscosity subsolution of \eqref{eq:nHJ} and $v\in LSC(\R^N)\cap L^\infty(\R^N)$ be a bounded viscosity supersolution of \eqref{eq:nHJ}. In order to show that $u\leq v$ in $\R^N$, we argue by contraction and assume that 
\[ M := \sup_{x\in\R^N}\big( u(x) -v(x)\big) >0.\]
In order to enhance the superlinear growth of the Hamiltonian, given by assumption (H2), we multiply the subsolution $u$ by a scaling parameter $\mu\in(0,1)$ and define, for all $x\in\R^N$, $\ovl{u}(x) = \mu u(x)$. Then, for $\mu$ sufficiently close to $1$, we will have that
\[ \ovl M := \sup_{x\in\R^N}\big( \ovl u(x) -v(x)\big) >0.\]
To ensure that a close supremum is to be attained, it is necessary to add a localization function $\psi_\beta$, defined as in Lemma \ref{lem.localization}, with a proper choice of constant $c>0$, so that, for $\beta$ sufficiently small, 
\[ M_{\beta} := \sup_{x\in\R^N}\left( \ovl u(x) -v(x)  - \psi_\beta(x) \right) >0.\]
We finally double the variables and take $\eps>0$ sufficiently small in order that
\[ M_{\eps,\beta} := \sup_{x,y\in\R^N}\left( \ovl u(x) -v(y) - \frac{|x -  y|^2}{\eps^2} - \psi_\beta(y) \right) >0.\]
In the sequel, we denote the function over which the supremum is taken by
\[ \phi_{\eps,\beta}(x,y):= \ovl u(x) -v(y) - \frac{|x -  y|^2}{\eps^2} - \psi_\beta(y).\]

Several technical and useful results about the behaviour of the maxima points and maxima values used throughout the proof, are summed up in Lemma \ref{lem.max}. 

\begin{lemma}\label{lem.max} 
Under the notations above, assume $M>0$. Then the following assertions hold.
\begin{enumerate}[label=(\roman*)]
\item For any $\mu\in(0,1)$, the function $\ovl{u} = \mu u \in USC(\R^N)\cap L^\infty(\R^N)$ 
	is a bounded viscosity subsolution of the following partial integro-differential equation
	\[\lambda \ovl{u}- \I_x \ovl u(x) + \mu H(x, \mu^{-1}D\ovl{u}) = 0 \quad  \mbox{in} \quad \R^N.\]
\item Let $ c =  2(||u|| _\infty + ||v || _\infty)$ in equation ~\eqref{psi_func}
	defining the localization function $\psi$. 
   	If $\mu \in (0,1)$ is sufficiently close to $1$, and $\beta\in (0,1)$ is small enough, then 
	\[ M_{\eps, \beta}\geq M_\beta\ge \frac{\ovl{M}}{2}\geq \frac{M}{4} > 0. \]
\item The supremum $M_{\eps, \beta}$ is achieved at some 
	$(x_{\eps, \beta}, y_{\eps, \beta}) \in {\ovl B_{C\eps + 2/\beta} \times \ovl B_{2/\beta}}$
	with $C>0$ a constant depending only on  $||u||_\infty$ and $||v||_\infty$.
\item For fixed $\beta\in(0,1)$,
	$\displaystyle  \frac{|  x_{\eps, \beta} - y_{\eps, \beta} | ^2}{\eps^2}  = o_\eps^\beta(1),$
         and uniformly in $\beta>0$
	$\displaystyle |  {x}_{\eps,\beta} -  y_{\eps,\beta} | = o_\eps(1).$
\item For fixed $\beta$, $M_{\eps, \beta} \rightarrow M_{\beta} $ as $\eps \searrow 0$, and in turn, 
	$M_{\beta} \rightarrow \ovl{M}$ as $\beta \searrow 0$.
\end{enumerate}
\end{lemma}
\begin{proof}[Proof of Lemma \ref{lem.max}]
(i) In view of assumption (H2), the proof is immediate and identical to the local case, 
	due to the fact that the nonlocal term is linear with respect to $u$.
	
(ii) Note that $M \leq ||u || _\infty + ||v|| _\infty$, and
	\[ \ovl u(x) - v(x) = (u-v)(x) + (\mu - 1)u(x) \ge u(x) - v(x) - (1-\mu)|| u|| _\infty.\]
	Thus, taking supremum over all $x\in\R^N$, it follows that, for $\mu$ sufficiently close to $1$,
	\[ \ovl M \ge M - (1-\mu) | u| _\infty \ge \frac{M}{2}.\]
	The localization function $\psi_\beta$, defined as in Lemma \ref{lem.localization} 
	with $ c =  2(||u|| _\infty + ||v || _\infty)$, satisfies
	\begin{eqnarray*}
	 \psi_\beta(x)  > ||  u || _\infty + ||  v || _\infty && \text{ when } |  x |  \ge 2/\beta,\\
	 \psi_\beta(x)  = 0 & & \text{ when }  | x | \le 1/\beta.
	\end{eqnarray*}
	Let $\hat x\in \R^N$ be such that  $\ovl u(\hat x) - v(\hat x) \ge \ovl M/2$. 
	Then, for $\beta$ sufficiently small, $\psi_\beta (\hat x) = 0$ and 
	\[  \ovl M\ge  M_\beta \ge  \ovl u(\hat x) - v(\hat x) \ge \frac{\ovl M}{2}.\] 
	The supremum $M_{\eps, \beta}$, satisfies, for all $x\in\R^N$,
	\[ M_{\eps, \beta} \ge \phi_{\eps,\beta}(x,x) =  \ovl u(x) -v(x)- \psi_\beta(x),\]
	which implies, taking the supremum over all $x\in\R^N$ that
	$ \displaystyle M_{\eps, \beta}  \ge M_\beta \ge \frac{\ovl M}{2}.$
	
(iii) Note first that, for $y\in\R^N$ such that $|  y |  \ge 2/\beta$, we have
	$\psi_\beta(y) > ||  u || _\infty +  ||  v || _\infty$, and hence 
	\[\phi_{\eps,\beta}(x,y) =  \ovl u(x) - v(y) - \frac{|  x - y |^2 }{\eps^2}  - \psi_\beta(y) 
		\leq || u|| _\infty +  || v || _\infty  - \psi_\beta(y) < 0. \]
	Therefore the supremum is taken over $\R^N\times \ovl B_{2/\beta}$. 
	On the other hand, by the positivity of $M_{\eps,\beta}$, it follows that,
	 for any sequence $(x^n_{\eps,\beta},y^n_{\eps,\beta})\in\R^N\times \ovl B_{2/\beta}$ 
	 for which the sequence of values is converging to the maxima, i.e.
	 $ \phi_{\eps,\beta} (x^n_{\eps,\beta},y^n_{\eps,\beta}) \to M_{\eps,\beta}$, as $\eps\to 0$,
	we have
	\[ \frac{|  x^n_{\eps,\beta} - y^n_{\eps,\beta} |^2 }{\eps^2} \leq 
		\ovl u(x^n_{\eps,\beta} ) - v(y^n_{\eps,\beta}) - \psi_\beta(y^n_{\eps,\beta}) 
		\leq ||  u|| _\infty + ||  v|| _\infty = : C^2. \]
	This further implies that the sequence is bounded
	\[ |  x^n_{\eps,\beta} | \ \le 
	 	|  x^n_{\eps,\beta} - y^n_{\eps,\beta}|  +
		|  y^n_{\eps,\beta} |  \leq C\eps +2/ \beta,\]
	and hence there exists a subsequence converging to a point 
	$(x_{\eps,\beta}, y_{\eps,\beta})\in {\ovl B_{C\eps + 2/\beta} \times \ovl B_{2/\beta}}$, 
	as $n\to\infty$. In view of the upper semi-continuity of the function $\phi_{\eps,\beta}$, 
	it follows that the supremum $M_{\eps,\beta}$ is attained at $(x_{\eps,\beta}, y_{\eps,\beta})$.	
	
(iv)	It is immediate to see that, for fixed $\beta>0$, the sequence $M_{\eps,\beta}$ is monotone 
	with respect to $\eps>0$. Indeed, if $\eps \le \eps'$, then 
	$\phi_{\eps,\beta} (x,y) \le \phi_{\eps',\beta}(x,y)$, for all $x,y\in\R^N$, which in turn implies 
	\[M_\beta \le M_{\eps,\beta} \le M_{\eps',\beta}.\]
 	Therefore the sequence $\left( M_{\eps, \beta} \right)_{\eps}$ is decreasing as $\eps\searrow0$
	and bounded from below, so there exists $\widetilde M_\beta \ge M_\beta$ such that 
	\[ \lim_{\eps \rightarrow 0} M_{\eps, \beta} =  \widetilde{M}_\beta. \]
	On the other hand,  note that 
	\[ M_{2\eps, \beta} \geq \Phi_{2\eps,\beta}(x_{\eps,\beta}, y_{\eps,\beta})
	 	= M_{\eps,\beta} + \dfrac{3}{4\eps^2} |  {x}_{\eps,\beta} -  y_{\eps,\beta} | ^2, \]
	which implies that
	\[ \dfrac{3}{4\eps^2} |  {x}_{\eps,\beta} -  y_{\eps,\beta} | ^2 \leq M_{2\eps, \beta} - M_{\eps, \beta} \leq C^2.\]
	It follows from the convergence of $\left( M_{\eps, \beta} \right)_{\eps}$ as $\eps\searrow 0$  that
	\[ \dfrac{1}{\eps^2} |  {x}_{\eps,\beta} -  y_{\eps,\beta} |^2 = o_\eps^\beta(1) 
	\quad \text{ and } \quad 
	|  {x}_{\eps,\beta} -  y_{\eps,\beta} | = o_\eps(1). \]
	
(v) Note that, in view of (iii), the sequence $\left ({x}_{\eps,\beta},  y_{\eps,\beta} \right)_{\eps}$ is bounded,  
	and, in view of point (iv), there exists  $x_\beta \in \R^N$ such that, 
	up to a subsequence, ${x}_{\eps,\beta},  y_{\eps,\beta} \rightarrow x_\beta$ as $\eps\to 0$. 
	Since $u \in USC(\R^N)$ and $v \in LSC(\R^n)$, it follows that
	\begin{align*}
		& \limsup_{\eps\to 0} \left( \ovl{u}(x_{\eps,\beta}) -  v(y_{\eps,\beta}) \right) \leq
		 \ovl{u}(x_\beta) -v(x_\beta).
	\end{align*}
	This leads to 
	\begin{eqnarray*}
	M_{\beta} \leq  \widetilde M_\beta 
		& =  & \lim_{\eps\to 0}  M_{\eps,\beta}
		   =   \lim_{\eps\to 0} \phi_{\eps,\beta}({x}_{\eps,\beta}, y_{\eps,\beta})\\	   
		& \leq &\limsup_{\eps\to 0}\left(\ovl{u}(x_{\eps,\beta})-v(y_{\eps,\beta})-\psi_\beta(y_{\eps,\beta})\right)\\
		& \leq & \ovl{u}(x_\beta)- v(x_\beta)-\psi_\beta (x_\beta)
	\leq M_{\beta}.
	\end{eqnarray*}
	In conclusion, $\lim_{\eps \rightarrow 0} M_{\eps, \beta} = M_{\beta}$.
	In order to pass now $\beta\to 0$, recall that  $M_{\beta} \leq \ovl{M}$ for every $\beta > 0$.
	From the definition of $\ovl M$, it follows that, for any $\eta > 0$, there exists $x_{\eta} \in \R^N$ such that 
	\[ \ovl{u}(x_{\eta}) - v(x_{\eta}) > \ovl{M} - \eta. \]
	From the definition of $M_{\beta}$, we infer that
	\[ M_{\beta} \geq \ovl{u}(x_{\eta}) - v(x_{\eta}) - \psi_{\beta}(x_{\eta}) 
		\geq \ovl{M} - \eta - \psi_{\beta}(x_{\eta}).\]
	Choosing $\beta$ sufficiently small, we have $|  x_{\eta} |  < 1/\beta$, so that 
	$\psi_\beta(x_\eta) = 0$. 
	Therefore, we have
	\[ \ovl M \geq  M_\beta \geq  \ovl M   - \eta. \]
	This implies precisely that  $\lim_{\beta \rightarrow 0} M_{\beta} = \ovl M$.
\end{proof}

In view of~Lemma \ref{lem.max}, for $\mu$ sufficiently close to $1$ and $\beta$ sufficiently small, the supremum 
$M_{\eps,\beta}$ of the function $\phi_{\eps,\beta}$ is attained at some point $(x_{\eps,\beta},y_{\eps,\beta})$. To simplify the notation we hereafter drop the dependence on $\eps,\beta$ both on the function, simply denoted 
$\phi := \phi_{\eps,\beta}$, and on the maxima points, denoted by
$(\ovl x,\ovl y) : = (x_{\eps,\beta}, y_{\eps,\beta})$. 
We also introduce the simplified notation for the vectors
\[ \displaystyle \ovl p : = 2 \frac{\ovl{x}_{\eps, \beta} - \ovl y_{\eps, \beta}}{\eps^2}\quad \text{ and } \quad
   \displaystyle \ovl q : = - D\psi_\beta(y_{\eps,\beta}). \]
However, we bear in mind their dependence on $\eps,\beta$, which is employed at the end of the proof.

Thus, let $(\ovl x,\ovl y) $ be a global maximum point of the function $\phi$. Then, for all $x,y\in \R^N$, it holds
\begin{equation}\label{eq:max-ineq}
 	\ovl u(\ovl x) - v(\ovl y) - \frac{|\ovl x -  \ovl y|^2}{\eps^2} - \psi_\beta(\ovl y) \geq
 	\ovl u(x)        - v(y)       - \frac{|x -  y|^2}{\eps^2}               - \psi_\beta(y).
\end{equation}
Define, for fixed $\ovl y\in \R^N$, respectively for fixed $\ovl x\in\R^N$, the smooth functions 
\begin{eqnarray*}
	\varphi_{\ovl y}(x)   :=  \;  \frac{|x -  \ovl y|^2}{\eps^2},\quad
	\varphi_{\ovl x}(y)   :=  -  \frac{|\ovl x -  y|^2}{\eps^2} - \psi_\beta(y).
\end{eqnarray*}
In view of inequality~\eqref{eq:max-ineq}, it follows that $\ovl x$ is a global maximum point for $\ovl u - \varphi_{\ovl y} $ and $\ovl y$ is a global minimum point for $ v -\varphi_{\ovl x}$. This, together with Lemma~\ref{lem.max}(i), lead to the following viscosity inequalities for $\ovl u \in USC(\R^N)\cap L^\infty(\R^N)$ and $v\in LSC(\R^N)\cap L^\infty(\R^N)$ respectively:
\begin{eqnarray*}
\lambda \ovl{u}(\ovl{x}) 
	- \I[B_\delta](\ovl{x},  \varphi_{\ovl y}) - \I[B_\delta^c](\ovl{x}, \ovl{p}, \ovl{u}) 
	+ \mu H(\ovl{x},\mu^{-1}\ovl{p}) & \leq & 0, \\
\lambda v(\ovl y) 
	- \I[B_\delta](\ovl y, \varphi_{\ovl x}) - \I[B_\delta^c](\ovl y, \ovl{p} + \ovl{q}, v) 
	+ H(\ovl y, \ovl{p} + \ovl{q}) & \geq & 0.
\end{eqnarray*}
Subtracting the two inequalities, we obtain
\begin{eqnarray}
 \lambda (\ovl{u}(\ovl{x}) - v(\ovl y))  
 	& + & 	\mu H(\ovl{x}, \mu^{-1}\ovl{p}) -  H(\ovl y, \ovl{p} + \ovl{q}) \nonumber \\ 
	& \leq & 	\I[B_\delta](\ovl{x},  \varphi_{\ovl y})  - \I[B_\delta](\ovl y, \varphi_{\ovl x})  + 
			\I[B_\delta^c](\ovl{x}, \ovl{p}, \ovl{u}) - \I[B_\delta^c](\ovl y, \ovl{p} + \ovl{q}, v).
\label{eq:visc-ineq}
\end{eqnarray}
Choosing $\beta > 0$ small enough, we notice, in view of Lemma~\ref{lem.max}(ii), that
\begin{equation}\label{eq:est-fct}
	  \lambda (\ovl{u}(\ovl{x}) - v(\ovl y)) \geq \lambda M_{\eps, \beta}
          \geq \lambda M_{\beta} \geq \lambda  \frac{\ovl M}{2} \geq \lambda \frac{M}{4}.
\end{equation}
In the following, we estimate the difference terms appearing in inequality \eqref{eq:visc-ineq}.

\begin{lemma}[Estimate for the Hamiltonian difference]\label{lem.est-H}
Let $H$ satisfy assumptions (H1) and (H2). 
Then, for $\mu$ sufficiently close to $1$, it holds that
\begin{equation}\label{eq:est-H}
\mu H(\ovl{x},\mu^{-1}\ovl{p}) -  H(\ovl y, \ovl{p} + \ovl{q}) 
\geq \frac12(1 - \mu)b_m |\ovl{p} |^m - o_{\beta}(1) - o_{\eps}(1) - \lambda \frac{M}{8}.
\end{equation}
\end{lemma}

\begin{proof}[Proof of Lemma \ref{lem.est-H}]
In view of Lemma \ref{lem.localization}, we have $|\ovl q |= |D\psi_\beta (\ovl y) |\leq C_0\beta$. Hence, choosing $\beta$ sufficiently small, we have $|\ovl q |\le 1$. In view of assumptions (H1) and (H2) it follows that, for $\mu$ sufficiently close to $1$, and for $\eps>0$ small,
\begin{eqnarray*}
\mathcal{H} 
	& : = &   \left( \mu H(\ovl{x}, \mu^{-1}\ovl{p}) - H(\ovl{x}, \ovl{p}) \right)
		  + \left( H(\ovl{x}, \ovl{p}) - H(\ovl y, \ovl{p} + \ovl{q}) \right) \\	 
	&\geq & (1 - \mu) \big( b_m |\ovl{p} |^m - b_0\big) 
		  - \omega_H^1\big(|\ovl{x} - \ovl y |\big) \; \big(1 + |\ovl{p} |^m\big) 
		  - \omega_H^2\big(|\ovl{q} |\big)	\;\big(1 + |\ovl{p} |^{m - 1}\big) \\
	& =  &  \left( (1 - \mu) b_m - \omega_H^1\big(|\ovl{x} - \ovl y |\big) \right) |\ovl{p} |^m 
		- \omega_H^2\big(|\ovl{q} |\big)	\;\big(1 + |\ovl{p} |^{m - 1}\big)
		- \omega_H^1(|\ovl{x} - \ovl y |) - (1 - \mu)b_0.
	\end{eqnarray*}
By Young's inequality, we have
\[\omega_H^2(|\ovl{q} |) |\ovl{p} |^{m - 1}
		\leq  \frac1m \omega_H^2(|\ovl{q} |)^{m/2} + 
		\frac{m-1}{m}\omega_H^2(|\ovl{q} |)^{m/(2(m - 1))}  |\ovl{p} |^m , \]
which, up to modifying the moduli of continuity by a constant depending on $m$, gives 
\begin{eqnarray*}
\mathcal{H} 
	&\geq & \left( (1 - \mu) b_m - \omega_H^1\big(|\ovl{x} - \ovl y |\big) 
						  - \omega_H^2(|\ovl{q} |)^{m/(2(m - 1))}  \right) |\ovl{p} |^m \\ 
	&&	 -  \omega_H^2(|\ovl{q} |)^{m/2} -  \omega_H^2(|\ovl{q} |) 
		 - \omega_H^1(|\ovl{x} - \ovl y |) - (1 - \mu)b_0.
	\end{eqnarray*}
Let $\mu$ be sufficiently close to $1$ so that $(1 - \mu)b_0 \leq \lambda M/8. $
In view of Lemma \ref{lem.max}(iv),  $|\ovl{x} - \ovl y |\rightarrow 0$ as $\eps \searrow 0$ and $|\ovl{q} |\rightarrow 0$ as $\beta \searrow 0$. Consequently,  fix $\eps, \beta > 0$ small enough so that
\[ (1 - \mu)b_m  - \omega_H^1(|\ovl{x} - \ovl y |)   
			- \omega_H^2(|\ovl{q} |)^{m/(2(m - 1))} 
   \geq \frac12  (1 - \mu) b_m> 0. \]
Hence, we attain the following estimate
\begin{eqnarray*}
\mathcal{H} 
	& \geq & \frac12(1 - \mu)b_m |\ovl{p} |^m 
		-  \omega_H^2(|\ovl{q} |)^{m/2} -  \omega_H^2(|\ovl{q} |) 
		 - \omega_H^1(|\ovl{x} - \ovl y |) 
		 - \lambda \frac{M}{8} \\
	& = & \frac12(1 - \mu)b_m |\ovl{p} |^m - o_{\beta}(1) - o_{\eps}(1) - \lambda \frac{M}{8}.
\end{eqnarray*}		 
\end{proof}

\begin{lemma}[Estimate of the Nonlocal difference inside $B_\delta$]\label{lem.est-NL-in}
Let $(\nu_\xi)_{\xi\in\R^N}$ be a family of Lévy measures satisfying assumption (M1). Then
\begin{equation}\label{eq:est-NL-in}
 \I[B_\delta](\ovl{x},  \varphi_{\ovl y})  - \I[B_\delta](\ovl y, \varphi_{\ovl x})
 	 \leq \frac{1}{\eps^2}  o_\delta^{\ovl x,\ovl y}(1)  + o_\beta(1).
\end{equation}
\end{lemma}

\begin{proof}[Proof of Lemma \ref{lem.est-NL-in}]
Direct computations give, in view of assumption (M1) and of Lemma~\ref{lem.localization}
\begin{eqnarray*}
\; \I[B_\delta](\ovl{x},  \varphi_{\ovl y}) 
	& = & \; \int_{B_\delta} \left(\varphi_{\ovl y} (\ovl x+z) - \varphi_{\ovl y} (\ovl x)- 
		D\varphi_{\ovl y} (\ovl x)\cdot z \right) \nu_{\ovl x}(dz) \\
	& = & \; \frac{1}{\eps^2}\int_{B_\delta} \left( |\ovl x +z - \ovl y|^2 - |\ovl x - \ovl y|^2 - 
		2 (\ovl x - \ovl y)\cdot z \right) \nu_{\ovl x}(dz)
	 =  \frac{1}{\eps^2}\int_{B_\delta}|z|^2\nu_{\ovl x}(dz) 
	 \leq  {\frac{1}{\eps^2} o^{\ovl x}_\delta(1)}, \\
- \I[B_\delta](\ovl{y},  \varphi_{\ovl x}) 
	& = & - \int_{B_\delta} \left(\varphi_{\ovl x} (\ovl y+z) - \varphi_{\ovl x} (\ovl y)- 
		D\varphi_{\ovl x} (\ovl y)\cdot z \right) \nu_{\ovl y}(dz) \\
	& = & \frac{1}{\eps^2}\int_{B_\delta} \left( |\ovl x - \ovl y - z|^2 - |\ovl x - \ovl y|^2 + 
		2 (\ovl x - \ovl y)\cdot z \right) \nu_{\ovl y}(dz) + \I[B_\delta](\ovl{y},  \psi_\beta) 
	 \leq  {\frac{1}{\eps^2} o^{\ovl y}_\delta(1) + C \beta^2} .
\end{eqnarray*}
Summing up, we have
\[ \I[B_\delta](\ovl{x},  \varphi_{\ovl y})  - \I[B_\delta](\ovl y, \varphi_{\ovl x})
 	 \le \frac{1}{\eps^2} o_\delta^{\ovl x,\ovl y}(1)   +  o_\beta(1).\]
\end{proof}

\begin{lemma}[Estimate of the nonlocal difference outside $B$]\label{lem.est-NL-out}
Let $(\nu_\xi)_{\xi\in\R^N}$ be a family of Lévy measures satisfying assumption (M1).
Then, for  any $R>1$, the following holds
\begin{equation}\label{eq:est-NL-out}
 	\I[B^c](\ovl{x}, \ovl{u}) - \I[B^c](\ovl y, v) 	
		\leq 	2 \| u \|_\infty {|\nu_{\ovl x} - \nu_{\ovl y}|(B_R\setminus B)}+ o^{\ovl x, \ovl y}_{1/R}(1)+ 
			{ o^{\ovl y}_\beta(1)}.
\end{equation}
If, in addition, assumptions (M2) and (M3) hold, then 
for $\eps>0$ small enough, 
\begin{equation}\label{eq:est-NL-out-unif}
	\I[B^c](\ovl{x}, \ovl{u}) - \I[B^c](\ovl y, v) \leq 
 	\lambda  \frac{M}{16} + o_{\beta}(1).
\end{equation}
\end{lemma}


\begin{proof}[Proof of Lemma \ref{lem.est-NL-out}]
Direct computations give
\begin{eqnarray*}
 \I[B^c](\ovl{x}, \ovl{u})  -  \I[B^c](\ovl{y}, v) 
	& = &  \int_{B^c} \left(\ovl u (\ovl x+z) - \ovl u(\ovl x)\right) \nu_{\ovl x}(dz) 
		- \int_{B^c} \left(v (\ovl y +z) - v(\ovl y) \right) \nu_{\ovl y}(dz) \\	
	& = & \int_{B^c} \left(\ovl u (\ovl x+z) - \ovl u(\ovl x)\right) \left( \nu_{\ovl x}(dz) - \nu_{\ovl y}(dz)\right) \\
	&&  +  \; \int_{B^c}  \left(\ovl u (\ovl x+z) - \ovl u(\ovl x)  - (v (\ovl y +z) - v(\ovl y)) \right) \nu_{\ovl y}(dz).	
\end{eqnarray*}
However, in view of inequality \eqref{eq:max-ineq} and~Lemma~\ref{lem.localization}, we have
\begin{eqnarray}\label{eq:est-Bc}
\int_{B^c}\left(\ovl u(\ovl x+z)-\ovl u(\ovl x)-(v(\ovl y+z)-v(\ovl y))\right)\nu_{\ovl y}(dz)  &\leq & \I[B^c](\ovl y, \psi_\beta) = o^{\ovl y}_\beta(1).
\end{eqnarray}
In order to estimate the first term, we split the integration domain into $B_R\setminus B$ and $B_R^c$, and use the boundedness of $\ovl u$ to obtain that
\begin{eqnarray*}
\int_{B^c} \left(\ovl u (\ovl x+z) - \ovl u(\ovl x)\right) 
\left( \nu_{\ovl x}(dz) - \nu_{\ovl y}(dz)\right) 
& = &  \int_{B_R\setminus B}\left(\ovl u (\ovl x+z) - \ovl u(\ovl x)\right) 
					\left( \nu_{\ovl x}(dz) - \nu_{\ovl y}(dz)\right)\\
& + & \int_{B_R^c} \left(\ovl u (\ovl x+z) - \ovl u(\ovl x)\right)  \nu_{\ovl x}(dz) 
 	- \int_{B_R^c} \left(\ovl u (\ovl x+z) - \ovl u(\ovl x)\right)  \nu_{\ovl y}(dz)  \\
&\le & 2 \| u \|_\infty \left(  \int_{B_R\setminus B} |\nu_{\ovl{x}}- \nu_{\ovl y} | (dz)
	+   \int_{B_R^c}\nu_{\ovl{x}}(dz) +  \int_{B_R^c}\nu_{\ovl y}(dz)\right).
\end{eqnarray*}
Thanks to the regularity of the measure $\nu_\xi$ for fixed $\xi$, we get
\begin{eqnarray}\label{form428}
  \int_{B_R^c}\nu_{\ovl{x}}(dz) +  \int_{B_R^c}\nu_{\ovl y}(dz) \leq  o^{\ovl x}_{1/R}(1) + o^{\ovl y}_{1/R}(1),
&&
\end{eqnarray}
and~\eqref{eq:est-NL-out} follows.

Assuming now that (M2) holds, we obtain
a uniform decay at infinity with respect to the points 
${\ovl x, \ovl y}$, that is, 
the right-hand side of~\eqref{eq:est-Bc} reduces to $o_\beta (1)$ and
the right-hand side of~\eqref{form428} to $o_{1/R}(1)$.
Moreover, under Assumption (M3), if  $\omega_{R,1}$ is the modulus of continuity corresponding to the total variation distance between two Lévy measures on the circular crown $B_R\setminus B$, then we have
\begin{eqnarray*}
|\nu_{\ovl x} - \nu_{\ovl y}|(B_R\setminus B) \leq  \omega_{R,1}(|\ovl x - \ovl y|) \leq  o_\eps^{R}(1),
 &&
\end{eqnarray*} 
where the last inequality comes from  Lemma~\ref{lem.max}  (iv).
It follows that~\eqref{eq:est-NL-out} now reads
\begin{eqnarray*}
 \I[B^c](\ovl{x}, \ovl{u}) - \I[B^c](\ovl y, v) 
	& \leq &  o_\eps^{R}(1) +    o_{1/R}(1) +  o_\beta(1).
\end{eqnarray*}
Fixing $R\geq 1$ big enough so that
$o_{1/R}(1)\leq \lambda M/ 32$, then
taking $\eps>0$ small enough, such that  
$o_\eps^R(1) \leq \lambda M/32$, we conclude that~\eqref{eq:est-NL-out-unif} holds.

\end{proof}

\begin{lemma}[Estimate of the nonlocal difference on the crown $B_\rho\setminus B_\delta$]
\label{lem.est-NL-crown}
Let $(\nu_\xi)_{\xi\in\R^N}$ be a family of Lévy measures satisfying assumption (M1). 
Then, for all $\rho\in (\delta,1]$, the following holds
\begin{eqnarray*}
 	\I[B_{\rho}\setminus B_\delta](\ovl{x}, \ovl p, \ovl{u}) - 
	\I[B_{\rho}\setminus B_\delta](\ovl y, \ovl p + \ovl q, v) 
 		&\leq &\frac{C}{\eps^2}\left({W_2}(\nu_{\ovl x}, \nu_{\ovl y})(B_{\rho})\right)^2
		+ \frac{1}{\eps^2} o_{\delta}^{\ovl x,\ovl y}(1) + o_\beta(1).
\end{eqnarray*}
If, in addition, (M4) holds, then, for $\eps>0 $ and $\beta>0$ sufficiently small and for all $\rho\in(\delta,1]$,
\begin{equation}\label{eq:est-NL-crown-unif}
 	\I[B_{\rho}\setminus B_\delta](\ovl{x}, \ovl p, \ovl{u}) - 
	\I[B_{\rho}\setminus B_\delta](\ovl y, \ovl p + \ovl q, v) \leq 	
 	\frac{1}{\eps^2}  |\ovl x - \ovl y| o_\rho(1) + \frac{1}{\eps^2}o_{\delta}^{\ovl x,\ovl y}(1) + o_\beta(1).
\end{equation}
\end{lemma}

\begin{proof}[Proof of Lemma \ref{lem.est-NL-crown}]
In what follows, we write $A_{\rho,\delta} = B_{\rho} \setminus B_\delta$, and for each $z\in\R^N$ denote
\begin{equation*}
	\begin{split}
		\ell_{\ovl x}\ovl u (z) := & ~ \ovl{u}(\ovl x + z) - \ovl{u}(\ovl x ) - \ovl{p} \cdot z, \\
		\ell_{\ovl y} v(z) := & ~ v(\ovl y + z) - v(\ovl y) - (\ovl{p} + \ovl{q}) \cdot z. 
	\end{split}
\end{equation*}
Hence, we can write
\begin{eqnarray*}
\T(\ovl x, \ovl y) & := &  
	 \I[B_{\rho}\setminus B_\delta](\ovl{x}, \ovl p, \ovl{u}) - 
	 \I[B_{\rho}\setminus B_\delta](\ovl y, \ovl p + \ovl q, v)\\
	 & = &
	 \int_{B_{\rho}} \ell_{\ovl x} \ovl u(z) \chi_{A_{\rho,\delta}}(z) \nu_{\ovl x}(dz) - 
	 \int_{B_{\rho}} \ell_{\ovl y} v(z) \chi_{A_{\rho,\delta}}(z) \nu_{\ovl y}(dz). 
\end{eqnarray*}	 
Notice that for all $\delta > 0$, the functions  $\ell_{\ovl x}\ovl u\chi_{A_{\rho,\delta}}$ and 
$\ell_{\ovl y} v\chi_{A_{\rho,\delta}}$ have support away from the origin. 
Consider an admissible coupling $\gamma_{(\ovl x, \ovl y)} \in \text{Adm}_{B_{\rho}}\left(\nu_{\ovl x}^\rho, \nu_{\ovl y}^\rho\right)$. 
In view of Proposition~\ref{prop.split.f+g}  in the Appendix, it follows that
\begin{eqnarray*}
\T(\ovl x, \ovl y) 
& = & \int \limits_{B_{\rho} \times B_{\rho}}
	\left(\ell_{\ovl x} \ovl u(z_1) \chi_{A_{\rho,\delta}}(z_1) 
					 - \ell_{\ovl y} v(z_2) \chi_{A_{\rho,\delta}}(z_2)\right) d\gamma_{(\ovl x, \ovl y)}(z_1, z_2) \\
& =& \T_1(\ovl x, \ovl y) + \T_2(\ovl x, \ovl y) + \T_3(\ovl x, \ovl y), 
\end{eqnarray*}
where
\begin{eqnarray*}
\T_1(\ovl x, \ovl y) & :=  & \int \limits_{A_{\rho,\delta} \times A_{\rho,\delta}}
	 	\left(  \ell_{\ovl x} \ovl u(z_1)  - \ell_{\ovl y} v(z_2)\right) d\gamma_{(\ovl x, \ovl y)}(z_1, z_2) \\
\T_2(\ovl x, \ovl y)  & := &  \int \limits_{A_{\rho,\delta} \times B_\delta} 
		\ell_{\ovl x} \ovl u(z_1) d\gamma_{(\ovl x, \ovl y)}(z_1, z_2) \\
\T_3(\ovl x, \ovl y) & := &  - \int \limits_{B_\delta \times A_{\rho,\delta}} 
		\ell_{\ovl y} v(z_2) d\gamma_{(\ovl x, \ovl y)}(z_1, z_2).
\end{eqnarray*}
In the following, we estimate each of the integral-differential terms $\T_i(\ovl x,\ovl y)$, $i=1,2,3$ separately. 
\smallskip

In order to estimate $\T_1(\ovl x,\ovl y)$, we make use of inequality \eqref{eq:max-ineq}, applied to $x = \ovl x +z_1$ and $y = \ovl y + z_2$, which gives
\[ \ell_{\ovl x} \ovl u(z_1) - \ell_{\ovl y} v(z_2) \leq
	 \frac{1}{\eps^2} |z_1 - z_2|^2  + 
	 \left( \psi_\beta(\ovl y + z_2) - \psi_\beta(\ovl y) - D\psi_\beta(\ovl y) \cdot z_2 \right).  \]
Therefore, after integration, we get the following estimate
\begin{eqnarray*}
\T_1(\ovl x, \ovl y) 
        & \leq &  \int \limits_{A_{\rho,\delta} \times A_{\rho,\delta}}
                \left(  \frac{1}{\eps^2} |z_1 - z_2|^2  +
		\left( \psi_\beta(\ovl y + z_2) - \psi_\beta(\ovl y) - D\psi_\beta(\ovl y) \cdot z_2 \right) \right)
	 	d\gamma_{(\ovl x, \ovl y)}(z_1, z_2)\\
	& \leq &  \frac{1}{\eps^2} \int \limits_{A_{\rho,\delta} \times A_{\rho,\delta}} |z_1 - z_2|^2d\gamma_{(\ovl x, \ovl y)}(z_1, z_2) 		
		+ \|D^2\psi_\beta \|_\infty \int \limits_{A_{\rho,\delta} \times A_{\rho,\delta}}|z_2|^2d\gamma_{(\ovl x, \ovl y)}(z_1,z_2)\\
	& \leq & \frac{1}{\eps^2} \int \limits_{B_{\rho} \times B_{\rho}} 
			|z_1 - z_2|^2d\gamma_{(\ovl x, \ovl y)}(z_1, z_2) 		
		+ 4 C_0\beta^2 \int \limits_{B_{\rho}}|z_2|^2 \nu_{\ovl y}(d z_2) \\
	& \leq & \frac{1}{\eps^2} \int \limits_{B_{\rho} \times B_{\rho}} 
			|z_1 - z_2|^2d\gamma_{(\ovl x, \ovl y)}(z_1, z_2) 		
		+ o_\beta(1),
\end{eqnarray*}
where we have used  the positivity of the admissible plan $\gamma_{(\ovl x,\ovl y)}$ for the first term
and Proposition~\ref{lem.Gigli} for the estimate of the latter term.
\smallskip

In order to estimate $\T_2(\ovl x,\ovl y)$, we make use of inequality \eqref{eq:max-ineq}, applied to $x = \ovl x +z_1$ and $y = \ovl y$ and observe that
\begin{eqnarray*}
 	\ell_{\ovl x}\ovl u(z_1)
	& \leq  & \frac{|\ovl x -  \ovl y + z_1|^2}{\eps^2} -  \frac{|\ovl x -  \ovl y|^2}{\eps^2}  -\ovl p \cdot  z_1 =  \frac{|z_1|^2}{\eps^2}. 
\end{eqnarray*}
This implies, by integration and thanks again to Proposition~\ref{lem.Gigli}, that
\begin{eqnarray*}
\T_2(\ovl x, \ovl y)  
	& \leq & \frac{1}{\eps^2} \int \limits_{A_{\rho,\delta} \times B_\delta}  
				|z_1|^2 d\gamma_{(\ovl x, \ovl y)}(z_1, z_2) \\
	& \leq & \frac{2}{\eps^2} \int \limits_{A_{\rho,\delta} \times B_\delta} 
				|z_1-z_2|^2  d\gamma_{(\ovl x, \ovl y)}(z_1, z_2) 
	 	  + \frac{2}{\eps^2} \int \limits_{A_{\rho,\delta} \times B_\delta} 
		  		|z_2|^2 d\gamma_{(\ovl x, \ovl y)}(z_1, z_2) \\
	& \leq & \frac{2}{\eps^2} \int \limits_{B_{\rho} \times B_{\rho}} 
				|z_1-z_2|^2  d\gamma_{(\ovl x, \ovl y)}(z_1, z_2)
		  + \frac{8}{\eps^2} \int \limits_{B_\delta} |z_2|^2 \nu_{\ovl y}(dz_2)\\
	& \leq & \frac{2}{\eps^2} \int \limits_{B_{\rho} \times B_{\rho}} 
				|z_1-z_2|^2  d\gamma_{(\ovl x, \ovl y)}(z_1, z_2)
		  + \frac{1}{\eps^2} o_{\delta}^{\ovl y}(1).	 		  	 
\end{eqnarray*}

Similarly, in order to estimate $\T_3(\ovl x,\ovl y)$, we make use of inequality \eqref{eq:max-ineq}, applied to $x = \ovl x$ and $y = \ovl y + z_2$ and observe that
\begin{eqnarray*}
 	- \ell_{\ovl y} v(z_2) 
	& \leq  & \frac{|z_2|^2}{\eps^2}  +  \left( \psi_\beta(\ovl y + z_2) - \psi_\beta(\ovl y) - D\psi(\ovl y) \cdot z_2 \right). 
\end{eqnarray*}
This implies, by integration, using similar arguments to those used in the estimates above, that
\begin{eqnarray*}
\T_3(\ovl x, \ovl y)  
	& \leq & \frac{1}{\eps^2} \int \limits_{B_\delta \times A_{\rho,\delta}}  
				|z_2|^2 d\gamma_{(\ovl x, \ovl y)}(z_1, z_2)  +
	 \int \limits_{B_\delta \times A_{\rho,\delta}}  \left( \psi_\beta(\ovl y + z_2) - \psi_\beta(\ovl y) 
	 			+ \ovl q \cdot z_2 \right)d\gamma_{(\ovl x, \ovl y)}(z_1, z_2) \\
	& \leq & \frac{2}{\eps^2} \int \limits_{B_{\rho} \times B_{\rho}} |z_1-z_2|^2  d\gamma_{(\ovl x, \ovl y)}(z_1, z_2) 
	 	  + \frac{1}{\eps^2} o_{\delta}^{\ovl x}(1)
		  +  o_\beta(1). 
\end{eqnarray*}
In conclusion, we obtain
\begin{eqnarray*}
\T(\ovl x, \ovl y)  
& \leq &  \frac{5}{\eps^2} \int \limits_{B_{\rho} \times B_{\rho}} 
					|z_1 - z_2|^2d\gamma_{(\ovl x, \ovl y)}(z_1, z_2) 
		+ \frac{1}{\eps^2}  o_{\delta}^{\ovl x, \ovl y}(1)+ o_\beta(1).		  
\end{eqnarray*}
Taking infimum among all admissible couplings, we are led to the estimate
\begin{eqnarray*}
\T(\ovl x, \ovl y)  & \leq &   
	\frac{5}{\eps^2}  \inf_{\gamma_{(\ovl x, \ovl y)}}
	\int \limits_{B_{\rho} \times B_{\rho}} |z_1 - z_2|^2d\gamma_{(\ovl x, \ovl y)}(z_1, z_2) 		
		  + \frac{1}{\eps^2}  o_{\delta}^{\ovl x,\ovl y}(1)+ o_\beta(1) \\
	&  = &   \frac{5}{\eps^2}  \left(W_2(\nu_{\ovl x}, \nu_{\ovl y})(B_{\rho})\right)^2
		  + \frac{1}{\eps^2}  o_{\delta}^{\ovl x,\ovl y}(1)+ o_\beta(1).
\end{eqnarray*} 
Finally, when assumptions (M4) is employed, we obtain~\eqref{eq:est-NL-crown-unif}.
\end{proof}

We next establish, from all of the above estimates, the boundedness of $\ovl p$, for sufficiently small $\eps,\beta$ and $\mu$ sufficiently close to $1$.

\begin{lemma}\label{lem.p-bound}
Under the assumptions of the theorem, we have for $\eps>0,\beta>0$ sufficiently small and $\mu\in(0,1)$ sufficiently close to $1$, that
\begin{equation}\label{eq:p-bound}
	|\ovl{p} |\leq \dfrac{C}{(1 - \mu)^{1/(m - 1)}},
\end{equation}
where $C$ is a constant depending on $m$ and $\|u\|_\infty, \|v\|_\infty$, but on none of the parameters above.
\end{lemma}

\begin{proof}[Proof of Lemma \ref{lem.p-bound}]
It follows from the viscosity inequality \eqref{eq:visc-ineq}, 
and the estimates provided in  \eqref{eq:est-fct},  \eqref{eq:est-H}, 
\eqref{eq:est-NL-in},  \eqref{eq:est-NL-out-unif},  \eqref{eq:est-NL-crown-unif} with $\rho=1$, that 
\begin{eqnarray*} 
 \lambda \frac{M}{16} + \frac12(1 - \mu)b_m |\ovl{p} |^m  & \leq &  
 	\frac{C}{\eps^2}  |\ovl x - \ovl y| + 
    \frac{1}{\eps^2}  o_\delta^{\ovl x,\ovl y}(1)  +  
    o_{\eps}(1) + o_{\beta}(1).
 \end{eqnarray*} 
Recalling that $|\ovl p|= 2 |\ovl x - \ovl y| /\eps^2$, it follows that
\begin{eqnarray*} 
 \lambda \frac{M}{8} +  (1 - \mu)b_m |\ovl{p} |^m  
 	& \leq & C|\ovl p| + \frac{1}{\eps^2}  o_\delta^{\ovl x,\ovl y}(1) +  o_{\eps}(1) + o_{\beta}(1).
\end{eqnarray*}
Fix $\eps>0, \beta>0$ small enough, so that $o_{\eps}(1) <\lambda M/32$ and $o_{\beta}(1) <\lambda M/32$, we obtain
\begin{eqnarray*} 
 \lambda \frac{M}{16} +  (1 - \mu)b_m |\ovl{p} |^m  
 	& \leq & C|\ovl p| + \frac{1}{\eps^2}  o_\delta^{\ovl x,\ovl y}(1).
\end{eqnarray*}
Letting now $\delta \searrow 0$, it follows that
\begin{eqnarray*} 
 \lambda \frac{M}{16}
 	& \leq & |\ovl p| \left( C - (1 - \mu)b_m |\ovl{p} |^{m-1}\right).
\end{eqnarray*}
Since {$m>1$} and we assumed $M>0$, up to a modification of the constant $C$, we infer that
\[ |\ovl{p} |\leq \dfrac{C}{(1 - \mu)^{1/(m - 1)}}.\]
Note that the constant $C$ is independent of all the parameters taken within the proof, and it only depends on the data of the problem.

\end{proof}

To reach the conclusion, we provide a refined estimate of the nonlocal difference on the circular crown, which takes into account the boundedness of $\ovl p$. 

\begin{lemma}[Refined estimate of the nonlocal difference on the crown $B\setminus B_\delta$]\label{lem.est-NL-crown-2}
Let $(\nu_\xi)_{\xi\in\R^N}$ be a family of Lévy measures satisfying assumption (M1). 
Then, for any $\rho\in(\delta,1)$ the following holds
\begin{eqnarray}\label{eq:est-NL-crown-2a}
    \I[B\setminus B_\delta](\ovl{x}, \ovl p, \ovl{u}) - 
    \I[B\setminus B_\delta](\ovl y, \ovl p + \ovl q, v) 
        & \leq & \left( 2 \| u \|_\infty + |\ovl p| \right) 
                |\nu_{\ovl x}- \nu_{\ovl y}| (B\setminus B_{\rho}) +\\
  \nonumber  		
  &&  \frac{C}{\eps^2}\left({W_2}(\nu_{\ovl x}, \nu_{\ovl y})(B_{\rho})\right)^2
		  + \frac{1}{\eps^2} o_{\delta}^{\ovl x,\ovl y}(1)+ o_\beta(1).
\end{eqnarray}
If, in addition, assumptions (M2), (M3) and (M4) hold, then for all $0<\delta <\rho<1$  and for all $\eps,\beta>0$, 
\begin{equation}\label{eq:est-NL-crown-22}
 	\I[B\setminus B_\delta](\ovl{x}, \ovl p, \ovl{u}) - \I[B\setminus B_\delta](\ovl y, \ovl p + \ovl q, v) 
 		\leq C   \omega_{1,\rho}(|\ovl x - \ovl y|) + o_\rho(1) + \frac{1}{\eps^2}o_{\delta}^{\ovl x,\ovl y}(1)+ o_\beta(1).
\end{equation}
\end{lemma}

\begin{proof}[Proof of Lemma \ref{lem.est-NL-crown-2}]
In view of  Lemma \ref{lem.est-NL-crown}, we already have an estimate for the nonlocal difference on the circular crown $B\setminus B_\delta$, with small $\delta$. 
The goal is to refine the above estimate by considering an intermediate radius $\rho\in (\delta,1)$ and control the nonlocal difference by the total variation distance of the Lévy measures on the exterior crown $B\setminus B_\rho$ {-- taking profit of the boundedness of  $\ovl p$}, and by the Wasserstein distance of the Lévy measures on the inner crown $B_\rho \setminus B_\delta$. 

As such, the nonlocal difference on $B\setminus B_\rho$ is given by
\begin{eqnarray*}
 \T[B\setminus B_\rho](\ovl x, \ovl y)  & = & \I[B\setminus B_\rho](\ovl{x}, \ovl p, \ovl{u}) - \I[B\setminus B_\rho](\ovl y, \ovl p + \ovl q, v) \\
	& = & \int_{B\setminus B_\rho} \left(\ovl u (\ovl x+z) - \ovl u(\ovl x) - \ovl p \cdot z \right) 
		\left( \nu_{\ovl x}(dz) - \nu_{\ovl y}(dz)\right) \\
	&&  +  \; \int_{B\setminus B_\rho}  \left(\ovl u (\ovl x+z) - \ovl u(\ovl x)  - (v (\ovl y +z) - v(\ovl y)) + \ovl q \cdot z\right) \nu_{\ovl y}(dz).
\end{eqnarray*}
On one hand, we estimate the first term by employing the total variation of the two measures
\begin{eqnarray*}
 \int_{B\setminus B_\rho} \left(\ovl u (\ovl x+z) - \ovl u(\ovl x) - \ovl p \cdot z \right) 
		\left( \nu_{\ovl x}(dz) - \nu_{\ovl y}(dz)\right) 
	& \leq & \left( 2 \| u \|_\infty + |\ovl p| \right) \int_{B\setminus B_\rho} |\nu_{\ovl x}- \nu_{\ovl y}|(dz).
\end{eqnarray*}	
On the other hand, we control the second term by the integro-differential term associated with the localization function $\psi_\beta$:
\begin{eqnarray*}
 && \int_{B\setminus B_\rho}  \big(\ovl u (\ovl x+z) - \ovl u(\ovl x)  
  -  (v (\ovl y +z) - v(\ovl y)) + \ovl q \cdot z\big) \nu_{\ovl y}(dz) \\
 && \leq    \int_{B\setminus B_\rho} \left(\psi_\beta(\ovl y+z) - \psi_\beta(\ovl y) - D\psi_\beta(\ovl y) \cdot z \right) \nu_{\ovl y}(dz)
  \leq   |  D^2\psi_\beta| _\infty \int_{B\setminus B_\rho}|  z|  ^2 \nu_{\ovl y}(dz)  = o_\beta(1).
\end{eqnarray*}
Thus, we get the upper bound
\begin{eqnarray*}
 \T[B\setminus B_\rho](\ovl x, \ovl y) & \leq &  
 	\left( 2 \| u \|_\infty + |\ovl p| \right) |\nu_{\ovl x}- \nu_{\ovl y}|(B\setminus B_\rho) + o_\beta(1).
 \end{eqnarray*}
This, coupled with the estimate of the nonlocal difference on $B_\rho\setminus B_\delta$ given by Lemma \ref{lem.est-NL-crown}, leads to the global estimate \eqref{eq:est-NL-crown-2a}. Moreover, if assumptions (M3) and (M4) hold, then, in view of the boundedness of $\overline p$, the estimate boils down to
\begin{eqnarray*}
\T[B\setminus B_\delta](\ovl x, \ovl y) 
 		& \leq & \left( 2 \| u \|_\infty + |\ovl p| \right) \omega_{1,\rho}(|\ovl x - \ovl y|) + 
		|\ovl p| o_\rho(1) + 
        \frac{1}{\eps^2} o_{\delta}^{\ovl x, \ovl y}(1) + 
        o_\beta(1)\\
		& \leq & C   \omega_{1,\rho}(|\ovl x - \ovl y|) + o_\rho(1) + \frac{1}{\eps^2}o_{\delta}^{\ovl x,\ovl y}(1)+ o_\beta(1).
\end{eqnarray*}
\end{proof}

Finally, it follows from the viscosity inequality \eqref{eq:visc-ineq}, 
and the estimates provided in  \eqref{eq:est-fct},  \eqref{eq:est-H}, 
\eqref{eq:est-NL-in},  \eqref{eq:est-NL-out-unif},  \eqref{eq:est-NL-crown-22}
for the terms therein, that 
\begin{eqnarray*} 
 \lambda \frac{M}{16} + \frac12(1 - \mu)b_m |\ovl{p} |^m  
 	& \leq &  C \omega_{1,\rho}(|\ovl x - \ovl y|) + o_\rho(1)  + \frac{1}{\eps^2}o_{\delta}^{\ovl x,\ovl y}(1)  
		      + o_{\eps}(1) + o_{\beta}(1).
 \end{eqnarray*} 
Let first $\delta\searrow 0$, to obtain for any $0<\rho<1$, that
\begin{eqnarray*} 
 \lambda \frac{M}{16} + \frac12(1 - \mu)b_m |\ovl{p} |^m  
 & \leq &  C \omega_{1,\rho}(|\ovl x - \ovl y|) + 
		 o_\rho(1)  + o_{\eps}(1) + o_{\beta}(1).
 \end{eqnarray*} 
Fix now $\rho>0$ sufficiently small so that $o_\rho(1) \leq \lambda M/32$. This leads to 
\begin{eqnarray*} 
 \lambda \frac{M}{32}  \quad \leq \quad  \lambda \frac{M}{32} + \frac12(1 - \mu)b_m |\ovl{p} |^m  
 & \leq & C \omega_{1,\rho}(|\ovl x - \ovl y|)  + o_{\eps}(1) + o_{\beta}(1).
 \end{eqnarray*} 
Finally, send $\eps\searrow 0$ and $\beta\searrow 0$ and recall that, in view of Lemma \ref{lem.max}, $|\ovl x - \ovl y| =o_\eps(1)$, to infer that the right-hand side of the inequality tends to zero, arriving thus at the contraction with the fact that $M>0$. Hence, the assumption made is false and $ u\le v \text{ all over } \R^N.$

\hfill$\Box$

\section{Extensions}\label{sec:extensions}

The above comparison result can be extended to elliptic partial integro-differential equations with degenerate diffusion, as well as to parabolic problems where the time dependence only appears in the coefficients of the Hamiltonian. We discuss the two extensions below.

\subsection{Elliptic PIDEs with second order terms}

Consider an elliptic integro--differential equation with possibly degenerate second--order diffusion, which takes the following form 
\begin{equation}\label{eq:nHJ-diff}
\lambda u  - F(x, Du, D^2u) - \mathcal{I}_xu (x) + H(x, Du) = 0  \quad \mbox{in} \quad  \R^N,
\end{equation}
where $F:\R^N\times \R^N\times S^N\to \R$ is nondecreasing with respect to the matrix variable $X\in \mathbb S^N$.  

For the second order case, we employ the definition of viscosity sub- and supersolution involving sub- and super-jets instead of  the derivatives of the test function $(D\phi, D^2\phi)$. This is necessary when proving the comparison result to deal with the doubling of variables technique, resolved by the local Jensen-Ishii lemma (see \cite{BI08}). For completeness, we recall the notions of semi-jets within the nonlocal framework.

\begin{definition}
Let $u:\R^N\rightarrow\R$ and $v:\R^N\rightarrow\R$ be respectively an upper--semicontinuous and a lower-semicontinuous function.
\begin{itemize}
\item[(i)] We call \emph{the superjet of $u$ at $x\in \R^N$} and we denote by $\mathcal J^{2,+}u(x)$ the set 
$$
\mathcal J^{2,+}u(x) = \left\{ (p,X)\in\R^N\times\mathbb S^N; \; u(x+z) \leq u(x) + p\cdot z + \frac12 Xz\cdot z + o(|z|^2)\right\}.
$$
We say $(p,X)$ is a limiting superjet of $u$ at $x$ if there exists a sequence of points $x_n\in\R^N$ and a family of superjets $(p_n,X_n)\in \mathcal J^{2,+}u(x_n)$ such that $(x_n,p_n,X_n)\to (x,p,X)$ and $u(x_n) \to u(x)$. We denote this closure by $\ovl{ \mathcal J^{2,+}}u(x)$.
\item[(ii)] We call \emph{the subjet of $v$ at $x\in \R^N$} and we denote by $\mathcal J^{2,+}v(x)$ the set 
$$
\mathcal J^{2,-}v(x) = \left\{ (p,X)\in\R^N\times\mathbb S^N; \; v(x+z) \geq v(x) + p\cdot z + \frac12 Xz\cdot z + o(|z|^2)\right\}.
$$
We say $(p,X)$ is a limiting subjet of $v$ at $x$ if there exists a sequence of points $x_n\in\R^N$ and a family of subjets $(p_n,X_n)\in \mathcal J^{2,-}v(x_n)$ such that $(x_n,p_n,X_n)\to (x,p,X)$ and $v(x_n) \to v(x)$. We denote this closure by $\ovl {\mathcal J^{2,-}}v(x)$.
\end{itemize}
\end{definition}

We are now in place to state the definition of viscosity solutions for the second-order case PIDEs.

\begin{definition}[Viscosity solution -- second order nonlocal]$\;$
 \begin{itemize}
 \item[(i)] A function $u\in USC(\R^N)\cap L^{\infty}(\R^N)$ 
    is {\emph a viscosity subsolution} of \eqref{eq:nHJ-diff} iff, for any test function $\phi\in C^2(\R^N)\cap L^\infty(\R^N)$, if $\ovl x$ is a global maximum of  $u-\phi$  on $\R^N$ and $(p, X)\in \mathcal J^{2,+}u(\ovl x)$ with $p=D\phi(\ovl x)$ and $X\leq D^2\phi(\ovl x)$, then for any $\delta >0$,
	\[ \lambda u(\ovl x ) - F(\ovl x, p, X)- 
        \I[B_\delta](\ovl{x},  \phi) - 
        \I[B_\delta^c](\ovl{x}, p, u) 
	+ H(\ovl{x}, p)  \leq  0. \]
 \item[(ii)] A function $u\in LSC(\R^N)\cap L^{\infty}(\R^N)$ 
    is {\emph a viscosity supersolution} of \eqref{eq:nHJ-diff} iff, for any test function $\phi\in C^2(\R^N)\cap L^\infty(\R^N)$, if $\ovl x$ is a global minimum of  $u-\phi$ on $\R^N$  and $(p, Y)\in \mathcal J^{2,-}u(\ovl x)$ with $p=D\phi(\ovl x)$ and $Y\geq D^2\phi(\ovl x)$,  then for any $\delta >0$,
	\[ \lambda u(\ovl x ) - F(\ovl x, p, Y)- 
        \I[B_\delta](\ovl{x},  \phi) - 
        \I[B_\delta^c](\ovl{x}, p, u) 
	+ H(\ovl{x}, p)  \geq  0. \]
 \item[(iii)] A function $u\in C(\R^N)\cap L^{\infty}(\R^N)$ 
    is {\emph a viscosity solution} of \eqref{eq:nHJ-diff} iff it is both a viscosity subsolution and a viscosity supersolution.
 \end{itemize}
\end{definition}

In order to overcome the difficulties imposed by the presence of the $x-$dependent datum both in the second-order terms and in the nonlocal terms, we need a stronger assumption than (M4). Namely, we make a continuity assumption on the second-order moment for the total variation.

\begin{itemize}
\item[(M4)'] There exists $0<r_0<1$ such that, for all $0<r<r_0$, the second moment for the total variation distance between L\'evy measures satisfies, for all $x,y\in\R^N$,
		\[  \int_{B_r} | z |^2  |\nu_x - \nu_y|(dz) \leq o_r(1)|x-y|.\]	
\end{itemize}			
			
In view of Proposition \ref{prop.Wass-TV}, assumption (M4) is satisfied whenever (M4)' holds and all the estimates for the nonlocal terms obtained in the previous section are still valid.\smallskip

In addition to the set of assumptions (M1)--(M4)' on the family of Lévy measures, and (H1)-(H2) on the Hamiltonian, we shall assume the following on the local diffusion -- see~\cite{CIL92}.

\begin{itemize}
\item[(E)] For any $R>0$ there exists  a modulus of continuity $\omega^R_F:\R_+\to\R_+$, 
such that, for all $x,y\in B_R$, any $\eps, \alpha, \beta>0$ and for all  $X,Y\in\mathbb S^{N}$ satisfying the inequality
\begin{equation}\label{eq:matrix_ineq}
\begin{bmatrix} 
X &  0 \\
0 & -Y 
\end{bmatrix} 
\leq \frac{1}{\eps^2}
\begin{bmatrix} 
 I & -I \\
-I &  I 
\end{bmatrix} + o_\beta(1) + o_\alpha(1),
\end{equation}
it holds 
\begin{equation}\label{eq:ell-F} 
 F(x,\frac{x-y}{\eps^2}, X) - F(y,\frac{x-y}{\eps^2} + o_\beta(1), Y)\leq
\omega^R_{F}\left(|x-y|+\frac{|x-y|^{2}}{\eps^2}\right) + o_\beta(1) + o_\alpha(1).
\end{equation}
\end{itemize}

Before proving the comparison principle, we check beforehand a convergence result for smooth perturbations
of the nonlocal diffusion on small balls, which is necessary to employ the nonlocal Jensen-Ishii lemma
from ~\cite{BI08}. This result uses strongly the assumption (M4)'.

\begin{lemma}\label{lem.est-NL-alpha}
Assume (M1) and (M4)' hold.
Let $x\in\R^N$ and $\varphi\in C^2(B(x))\cap L^\infty(\R^N)$.  If $(x_n)_n$ is a sequence of points in $\R^N$ converging to $x\in\R^N$ as $n\to\infty$, and $(\varphi_n)_n$ is a family of functions in $C^2(B(x))\cap L^\infty(\R^N)$ converging in $C_{\rm loc}^2(B(x))$ to $\varphi$, then for any fixed $\delta\in(0,1)$,
\[ \I[B_\delta](x_n,\varphi_n) = \I[B_\delta](x,\varphi) + o_{1/n}(1).\]
\end{lemma}

\begin{proof}
In view of the regularity of the functions $\varphi_n$ and $\varphi$, direct computations give
\begin{eqnarray*}
    \I[B_\delta](x_n,\varphi_n) - \I[B_\delta](x,\varphi) 
    & = & \int_{B_\delta} \big( \varphi_n(x_n+z) - \varphi_n(x_n) - D\varphi_n(x_n) \big) \nu_{x_n}(dz)\\
    && - \int_{B_\delta} \big( \varphi(x+z) - \varphi(x) - D\varphi(x) \big) \nu_{x}(dz)\\
    & = & \frac{1}{2} \int_{|z|\leq \delta}\int_0^1 
    			\Big( D^2\varphi_n(x_n+\theta z) - D^2\varphi(x+\theta z) \Big) z\cdot z \;d\theta \;\nu_{x_n} (dz)\\
   &&	- \frac{1}{2} \int_{|z|\leq \delta}\int_0^1 
    			D^2\varphi(x+\theta z)z\cdot z\; d\theta \; \big(\nu_{x_n} (dz) - \nu_x(dz)\big).
\end{eqnarray*}
From here, in view of assumptions (M1) and (M4)', we obtain that, for any $\delta >0$, and for $n$ sufficiently large,
\begin{eqnarray*}
   | \I[B_\delta](x_n,\varphi_n) - \I[B_\delta](x,\varphi) |
    & \leq & \frac{1}{2} | \varphi_n - \varphi|_{C^2(\overline {B_{2\delta}(x)})} \int_{B_\delta} | z |^2 \nu_{x_n} (dz) \\
   && + \frac{1}{2}| D^2\varphi|_{L^\infty(\ovl {B(x)})} \int_{B_\delta} | z|^2 |\nu_{x_n} - \nu_x| (dz)\\
      & \leq & \frac12 C_\nu | \varphi_n - \varphi|_{C^2(\ovl {B_{2\delta}(x)})}  +  | D^2\varphi|_{L^\infty(B(x))} o_\delta(1) | x_n- x|,
\end{eqnarray*}
and the conclusion follows.
\end{proof}

\begin{theorem}[Comparison Principle for Elliptic PIDEs with local diffusion]
\label{thm.comparison-diff}
Let $\lambda >0$ and $F$ be a diffusion nonlinearity satisfying the ellipticity assumption (E). Let $(\nu_\xi)_{\xi\in\R^N}\subset \mathcal M(\R^N)$ be a family of Lévy measures associated to the nonlocal operators $\left(\mathcal I_\xi(\cdot)\right )_{\xi\in\R^N}$ satisfying assumptions  (M1)--(M4)', and let $H$ be a Hamiltonian with superlinear growth, satisfying assumptions (H1)-(H2). 

Then, the comparison principle holds:  if $u\in USC(\R^N)\cap L^\infty(\R^N)$ is a viscosity subsolution of \eqref{eq:nHJ-diff} and  $v\in LSC(\R^N)\cap L^\infty(\R^N)$ is a viscosity supersolution of \eqref{eq:nHJ-diff}, 
then $u\leq v \text{ in } \R^N$.
\end{theorem}

\begin{remark}\em
The consideration of second-order terms requires the use of a nonlocal analog of Ishii-Jensen lemma. This is a key point, that we would like to discuss here briefly, providing references for the details that we omit.
The argument behind Ishii-Jensen lemma is the use use of inf- and sup-convolutions of $u$ and $v$ respectively, and then exploit the semiconvexity properties through Aleksandrov Theorem, see~\cite{CIL92}. Following this, if $\varphi \in C^2$ is a function for which $(x,y) \mapsto u(x) - v(y) - \varphi(x,y)$ has a maximum point at $(\bar x, \bar y)$, as in Proposition 3 in~\cite{BI08}, we define
\begin{align*}
u^\alpha(x, p) & = \sup_{|x - z| \leq 1} \Big{\{} u(z) + p(x - z) - \frac{1}{\alpha^2}|x - z|^2 \Big{\}}, \\
u_\alpha(y, q) & = \sup_{|y - z| \leq 1} \Big{\{} v(z) - q(y - z) + \frac{1}{\alpha^2}|y - z|^2 \Big{\}},
\end{align*}
where $p = D_x \phi(\bar x, \bar y)$, $q = -D_y \phi(\bar x, \bar y)$ and $\alpha > 0$. In order to use the viscosity inequalities for the convoluted functions, we need to consider the inf- and sup-convolution of the test function $\varphi$, which is already smooth. Then, in order to conclude Ishii-Jensen Lemma, we shall invoke Lemma~\ref{lem.est-NL-alpha} and for this, we need to prove that, if $\varphi \in C^2(\R^N)$, then 
$$\varphi^\alpha, \varphi_\alpha \to \varphi  \text{ in } C^2_{\rm loc}(\R^N).$$ 
Though it seems to be part of the folklore of viscosity solutions and convolutions, we discuss it briefly for completeness.
In general, if $\varphi \in C^2(\R^N)$ and $x \in \R^N$, taking $p = D\varphi(x)$ and $\alpha$ very small in terms of $\| Du \|_{L^\infty(B(x))}$, then 
$\displaystyle \sup_{x\in\R^N} \varphi^\alpha(x, p)$ is attained at a unique point  
$x + z^\alpha (x)$ with $|z^\alpha(x)| \in B_{1/2}$. This immediately implies that $\varphi^\alpha$ is $C^1(\R^N)$ and 
$ D\varphi^\alpha(x) = D\varphi(x + z^\alpha(x)), $
see the Introduction in~\cite{Barles94}. This, together with the first-order criteria that allows us to find $z^\alpha(x)$ and the Implicit Function Theorem, allows us to prove the map $x \mapsto z^\alpha(x)$ is in $C^1(\R ^N)$, and from here that $\varphi^\alpha$ is $C^2(\R ^N)$. Moreover, since we know that $z^\alpha(x) \to x$ as $\alpha \to 0$, from which we conclude that $\varphi^\alpha \to \varphi$ in $C^2_{\rm loc}(\R^N)$.
\end{remark}

\begin{proof}[Proof of Theorem~\ref{thm.comparison-diff}.]

We proceed as before and assume $M = \sup_{\R^N}(u-v)>0$, and consider as before, for $\eps, \beta>0$ and $\mu\in(0,1)$,
\[M_{\eps,\beta} = \sup_{\R^N\times \R^R}\big(\ovl u(x) - v(x) - \varphi(x,y)\big),\]
with 
$$\varphi(x,y) = \frac{| x-y|^2}{\eps^2} + \psi_\beta(y).$$
We next employ the nonlocal Jensen-Ishii lemma~\cite[Corollary 1.1]{BI08}. 

It follows that, for any $\delta >0$ there exists $\tilde \alpha>0$ such that, for $0<\alpha<\tilde \alpha$, there exists a family of limiting superjet and subjet $(p,X_\alpha)\in \ovl{\mathcal J^{2,+}} u(\ovl x)$, $(p+q,Y_\alpha)\in \ovl{\mathcal J^{2,-}} v(\ovl y)$ 
and a family of functions (given by sup-convolutions)
$$\varphi^\alpha(x,y) : = \sup_{| (x,y) - (x^*, y^*)| \leq 1} \left( \varphi(x^*, y^*) + p\cdot (x-x^*) + (p+q)\cdot (y-y^*) - \frac{1}{2\alpha} | (x,y)- (x^*, y^*)|^2\right)$$ 
satisfying
\begin{eqnarray*}
    p & = & D_x \varphi^\alpha(\ovl x, \ovl y) = D_x \varphi(\ovl x, \ovl y)= \frac{2}{\eps^2}(\ovl x-\ovl y)\\
    p+q & = & D_y \varphi^\alpha(\ovl x, \ovl y) =  - D_y \varphi(\ovl x, \ovl y) = \frac{2}{\eps^2}(\ovl x-\ovl y) - D\psi_\beta(\ovl y),
\end{eqnarray*}
and 
\[
\begin{bmatrix} 
X_\alpha &  0 \\
0 & -Y_\alpha 
\end{bmatrix} 
\leq D^2\varphi^\alpha(\ovl x,\ovl y) 
= D^2\varphi (\ovl x,\ovl y) + o_\alpha(1) 
= \frac{2}{\eps^2}
\begin{bmatrix} 
 I & -I \\
-I &  I 
\end{bmatrix} + o_\beta(1) + o_\alpha(1),
\]
such that the following viscosity inequalities hold
\begin{eqnarray*}
\lambda \ovl{u}(\ovl{x}) - F(\ovl x, p, X_\alpha)
	- \I[B_\delta](\ovl{x},  \varphi^\alpha(\cdot, {\ovl y})) 
    - \I[B_\delta^c](\ovl{x}, \ovl{p}, \ovl{u}) 
	+ \mu H(\ovl{x},\mu^{-1}\ovl{p}) & \leq & 0, \\
\lambda v(\ovl y) - F(\ovl y, p + q, Y_\alpha)
	- \I[B_\delta](\ovl y, -\varphi^\alpha({\ovl x},\cdot))
    - \I[B_\delta^c](\ovl y, \ovl{p} + \ovl{q}, v) 
	+ H(\ovl y, \ovl{p} + \ovl{q}) & \geq & 0.
\end{eqnarray*}
Subtracting the two inequalities, we obtain
\begin{eqnarray}  \label{eq:visc-ineq-2}
 \lambda (\ovl{u}(\ovl{x}) - v(\ovl y))  
 	& + & 	F(\ovl y, p + q, Y_\alpha) -  F(\ovl x, p, X_\alpha)
        +\mu H(\ovl{x}, \mu^{-1}\ovl{p}) -  H(\ovl y, \ovl{p} + \ovl{q}) \\ \nonumber
	& \leq & 	
         \I[B_\delta](\ovl{x},  \varphi^\alpha(\cdot, {\ovl y})) 
         -\I[B_\delta](\ovl y, - \varphi^\alpha({\ovl x},\cdot))
        + \I[B_\delta^c](\ovl{x}, \ovl{p}, \ovl{u}) 
        - \I[B_\delta^c](\ovl y, \ovl{p} + \ovl{q}, v).
\end{eqnarray}
From assumption (E), for $R = C\eps + 2/\beta$, we have
\begin{equation} \label{eq:est-F}
F(\ovl y, p + q, Y_\alpha) -  F(\ovl x, p, X_\alpha) \geq -\omega_F^R(| \ovl x-\ovl y| + \frac{1}{\eps^2}| \ovl x - \ovl y|^2) + o_\beta(1) + o_\alpha(1).
\end{equation}
Taking into account the previous remark, $\varphi^\alpha \to \varphi $ in $C^2_{\rm loc}(\ovl {B(\ovl x)}\times \ovl {B(\ovl y)})$ as $\alpha\to 0$. Applying Lemma~\ref{lem.est-NL-alpha}, we deduce that, all $\delta >0$ and for all $0<\alpha<\tilde \alpha $,
\begin{eqnarray}\label{eq:est-NL-alpha}
    \I[B_\delta](\ovl y, \varphi^\alpha({\ovl x},\cdot))
        - \I[B_\delta](\ovl{x}, - \varphi^\alpha(\cdot, {\ovl y}))
    = 
    \I[B_\delta](\ovl y, \varphi({\ovl x},\cdot))
        - \I[B_\delta](\ovl{x},  \varphi(\cdot, {\ovl y})) + o_\alpha(1).
\end{eqnarray}
Plugging \eqref{eq:est-F} and \eqref{eq:est-NL-alpha} into inequality \eqref{eq:visc-ineq-2} and letting $\alpha\to 0$, we obtain
\begin{eqnarray*}
 \lambda (\ovl{u}(\ovl{x}) - v(\ovl y)) &- &\omega_F^R(| \ovl x-\ovl y| + \frac{1}{\eps^2}| \ovl x - \ovl y|^2)
        +\mu H(\ovl{x}, \mu^{-1}\ovl{p}) -  H(\ovl y, \ovl{p} + \ovl{q}) \\ \nonumber
	& \leq & 	
         \I[B_\delta](\ovl{x},  \varphi(\cdot, {\ovl y})) 
         -\I[B_\delta](\ovl y, - \varphi({\ovl x},\cdot))
        + \I[B_\delta^c](\ovl{x}, \ovl{p}, \ovl{u}) 
        - \I[B_\delta^c](\ovl y, \ovl{p} + \ovl{q}, v). 
\end{eqnarray*}
Recalling that $| \ovl x-\ovl y| = o_\eps(1)$ and $| \ovl x-\ovl y|^2/\eps^2 = o_\eps^\beta(1)$, we arrive at 
\begin{eqnarray*}
 \lambda (\ovl{u}(\ovl{x}) - v(\ovl y)) 
 & + &  o_\eps(1) + o_\eps^\beta(1) + 
        \mu H(\ovl{x}, \mu^{-1}\ovl{p}) -  H(\ovl y, \ovl{p} + \ovl{q}) \\ \nonumber
 & \leq & 	
         \I[B_\delta](\ovl{x},  \varphi(\cdot, {\ovl y})) 
         -\I[B_\delta](\ovl y, - \varphi({\ovl x},\cdot))
        + \I[B_\delta^c](\ovl{x}, \ovl{p}, \ovl{u}) 
        - \I[B_\delta^c](\ovl y, \ovl{p} + \ovl{q}, v) .
\end{eqnarray*}
From here on, the proof follows exactly the lines of the one of Theorem~\ref{thm.comparison}.
\end{proof}

\subsection{Parabolic PIDEs}

We now discuss the case of time-dependent problems involving integro-differential operators  and superlinear and coercive Hamiltonians, of the form
\begin{equation}\label{eq:nHJ-t}
\begin{cases}
	\partial_t u(x, t) - \mathcal{I}_x u(x,t) + H(x, t, Du(x, t)) = 0 & \quad  \text{ in }  \R^N \times (0, T), \\
	u(x, 0) = u_0(x) & \quad \text{ in } \R^N.
\end{cases}
\end{equation}	 
We consider a nonlocal operator of Lévy type, with the Lévy measures $(\nu_\xi)_{\xi\in\R^N}\subset \mathcal M(\R^N)$ acting on the space-variable 
\[ \mathcal I_\xi u(x,t) = \int_{\R^N} \big( u(x + z,t) - u(x,t) - \1_B(z) Du(x,t) \cdot z \big) \nu_\xi(dz).\]
The case of time dependent  Lévy measures $(\nu_{(\xi,t)})_{(\xi,t)\in\R^N\times(0,T]}\subset \mathcal M(\R^N\times(0,T))$ could be partially addressed in this framework, by accordingly adapting the assumptions with t-dependence. However, to ensure completeness of this framework, we prefer to address this problem in a future work.

We assume the family of measures satisfies exactly the same set of assumptions as in the stationary case and, since we allow a time-dependent in the Hamiltonian, we upgrade to assumptions on $H$ to the case of space-time coefficients.

\begin{itemize}
\item[(H1-t)] $H:\R^N\times[0,T]\times\R^N\to\R$  is continuous on $\R^N\times[0,T]$ and 
		there exists $m>1$ and two moduli of continuity 
		$\omega^1_H,\omega^2_H: \R_+\to\R_+$ 
		such that, for all $t\in [0,T]$, and $x,y,p,q\in\R^N$, with $|q|\leq 1$,
		\[ H(y, t, p + q) - H(x, t, p) \leq 
			\omega_H^1(|x-y|)(1 + | p |^m) +
			\omega_H^2(| q |)(1 + | p |^{m - 1}). \]
\item[(H2-t)]  There exists $m>1$, some constants $b_m, b_0 >0$, $r_0>0$
		 and $\mu_0\in (0,1)$ such that for all $\mu\in[\mu_0,1]$, $t\in[0,T]$
		 and $x,p\in \R^N$ with $|p|\ge r_0$, 
		\[ \mu H(x, t,\mu^{-1} p) - H(x,t,p) \ge (1-\mu) (b_m | p |^m - b_0).\]	
\end{itemize}

\begin{theorem}[Comparison principle for time-dependent problems]\label{thm.comparison-time}
Let $(\nu_\xi)_{\xi\in\R^N}\subset \mathcal M(\R^N)$ be a family of Lévy measures associated with the nonlocal operators $\left(\mathcal I_\xi(\cdot)\right )_{\xi\in\R^N}$ satisfying assumptions  (M1)--(M4), and let $H$ be a Hamiltonian satisfying assumptions \text{(H1-t)} -  \text{(H2-t)}. 

Then, the parabolic comparison principle holds:  if $u\in USC(\R^N\times[0,T])\cap L^\infty(\R^N\times[0,T])$ is a viscosity subsolution of \eqref{eq:nHJ-t} and  $v\in LSC(\R^N\times[0,T])\cap L^\infty(\R^N\times[0,T])$ is a viscosity supersolution of \eqref{eq:nHJ-t}, so that 
$u(\cdot,0) \le u_0 \le v(\cdot,0)\; \text{ in } \;\R^N,$
with $u_0\in BUC(\R^N)$,
then 
$u \leq v \; \text{ in } \; \R^N\times [0,T].$
\end{theorem}

\begin{remark}\normalfont The notions of viscosity sub and supersolutions are defined in a similar manner to the stationary case. We require the sub- and supersolution to satisfy the equation in $\R^N\times(0, T)$ only, and not at the final time $T$. However, it can be checked, similarly to the first order case, that the notions of subsolutions and supersolutions can be considered to the extended set $ \R^N\times(0, T]$.
\end{remark}

\begin{proof} 
The key difference with the proof from the stationary case is that we need to double the variables both in space in time. In addition to this, we need to subtract a linear term in time, in order to overcome the absence of $u$ in the equation (replaced by the time derivative). Hence, we assume again that
$\displaystyle M  = \sup_{\R^N\times [0,T]}(u-v) >0$
and note that, for $\eps,\eta, \lambda, \beta>0$ sufficiently small, and $\mu$ close to $1$, we have that
\[ \overline M  : = \sup_{\R^{2N}\times [0,T]^2}\left( \ovl u(x,t) - v(y,s) - \phi(x,y,t,s) \right) >0,\]
with
\[ \varphi(x,y,t,s) = \frac{| x - y|^2}{\eps^2} + \frac{(t - s)^2}{\eta^2} + \psi_\beta(y) + \lambda t. \]
The maxima will be attained at some $(\ovl x, \ovl y, \ovl t, \ovl s)$ with $\ovl t, \ovl s >0$, and as $\eta \to 0$, they will converge to some $(\hat x, \hat y, t^*)$ and correspondingly we will have $\ovl p\to \hat p$, and $\ovl q\to \hat q$.
With the same notations as in the main comparison proof, we  are led to the following viscosity inequalities at the maxima points $(\ovl x, \ovl y, \ovl t, \ovl s)$:
\begin{eqnarray*}
 \lambda + \frac{2(\ovl t-\ovl s)}{\eta^2}
	- \I[B_\delta](\ovl{x},  \ovl t, \varphi_{\ovl y,\ovl s}) - \I[B_\delta^c](\ovl{x}, \ovl t, \ovl{p}, \ovl{u}) 
	+ \mu H(\ovl{x},\ovl t, \mu^{-1}\ovl{p}) & \leq & 0, \\
 \frac{2(\ovl t-\ovl s)}{\eta^2}
	- \I[B_\delta](\ovl y, \ovl s, \varphi_{\ovl x,\ovl t}) - \I[B_\delta^c](\ovl y, \ovl s, \ovl{p} + \ovl{q}, v) 
	+ H(\ovl y,  \ovl s, \ovl{p} + \ovl{q}) & \geq & 0.
\end{eqnarray*}
Subtracting the two inequalities, we obtain
\begin{eqnarray*}
 \lambda 
 	& + & 	\mu H(\ovl{x},\ovl t, \mu^{-1}\ovl{p}) -  H(\ovl y, \ovl s, \ovl{p} + \ovl{q}) \\ \nonumber
	& \leq & 	\I[B_\delta](\ovl{x}, \ovl t,  \varphi_{\ovl y, \ovl s})  - \I[B_\delta](\ovl y, \ovl s, \varphi_{\ovl x, \ovl t})  + 
			\I[B_\delta^c](\ovl{x}, \ovl t, \ovl{p}, \ovl{u}) - \I[B_\delta^c](\ovl y, \ovl s, \ovl{p} + \ovl{q}, v).
\end{eqnarray*}
Sending $\eta\to 0$, we have, in view of the continuity of $H$, that
\begin{eqnarray*}
 \lambda 
 	& + & 	\mu H(\hat x, t^*, \mu^{-1}\hat p) -  H(\hat y, t^*, \hat p + \hat q) \\ \nonumber
	& \leq & 	\I[B_\delta](\hat x, t^*,  \varphi_{\hat y, t^*})  - \I[B_\delta](\hat y, t^*, \varphi_{\hat x, t^*})  + 
			\I[B_\delta^c](\hat x, t^*, \hat p, \ovl{u}) - \I[B_\delta^c](\hat y, t^*, \hat p + \hat q, v).
\end{eqnarray*}
In view of (H1-t) and (H2-t) the estimate obtained in Lemma \ref{lem.est-H} still holds. The above inequality now takes a similar form to the one in the stationary case
\begin{eqnarray*}
  \frac{\lambda}{2}
 	& + & 	\frac12(1 - \mu)b_m |\hat p |^m - o_{\beta}(1) - o_{\eps}(1)\\ \nonumber
	& \leq & 	\I[B_\delta](\hat x, t^*,  \varphi_{\hat y, t^*})  - \I[B_\delta](\hat y, t^*, \varphi_{\hat x, t^*})  + 
			\I[B_\delta^c](\hat x, t^*, \hat p, \ovl{u}) - \I[B_\delta^c](\hat y, t^*, \hat p + \hat q, v).
\end{eqnarray*}
From here on, we can employ all the nonlocal estimates provided in the previous section and reach the conclusion.
\end{proof}

\subsection{Guillen-Mou-\'Swi\c{e}ch case}

In \cite{GMS19} the authors introduced optimal transport techniques to obtain comparison results for nonlocal equations with, a priori, no restriction on the order of the operators. They impose a Lipschitz condition with respect to the $p$--Wasserstein metric, i.e.,
\begin{itemize}
\item[(M4)''] There exists $p \in [1, 2]$ and $C_p>0$ such that, for all $x,y\in\R^N$,
                \[ W_p(\nu_{x}, \nu_{y})(B)\leq C_p |x-y|. \]
\end{itemize}
The exponent $p\in[1,2]$ is related to the singularity at $z=0$ for the Lévy measures: for a nonlocal operator of order $\sigma$, one has $p>\sigma$.
However, it is rather difficult to construct measures for which (M4)'' holds in the case $p>1$,
which in fact limits the use of such a hypothesis.
For example, in the case of measures with density 
\[\nu_\xi(z) = K(\xi,z) dz,\]
it has only been shown  that condition (M4)'' is satisfied when $p=1$, restricting the nonlocal operator to order $\sigma \in (0,1)$, see \cite[Corollary 4.5]{GMS19}. As a matter of fact, it is natural to expect $W_p$ to be only $1/p-$H\"older continuous, provided the kernel $K$ is Lipschitz in $\xi$. See \cite[Example 5.12]{GMS19} and Section \S~\ref{ex:kernel} for details. In addition, the authors show that the comparison principle still holds  if the Wasserstein distance is $1/2-$ H\"older continuous, provided the subsolution or the supersolution is $C^1$. 
Within our approach, we establish the comparison principle for $1/2-$ H\"older continuous $p$-Wasserstein distances, without having to assume a priori regularity on either the sub/supersolution.

\begin{theorem}[Comparison Principle -- $W_p$] \label{thm.comparison-Wp}
Let $\lambda >0$, $(\nu_\xi)_{\xi\in\R^N}\subset \mathcal M(\R^N)$ be a family of Lévy measures associated with the nonlocal operators $\left(\mathcal I_\xi(\cdot)\right )_{\xi\in\R^N}$ satisfying assumptions  (M1), (M2), (M3) with $r=1$ and (M4)'' with $p>1$, and let the  Hamiltonian $H$ satisfy assumptions (H1)-(H2). 

Then, the comparison principle holds:  if $u\in USC(\R^N)\cap L^\infty(\R^N)$ is a viscosity subsolution of \eqref{eq:nHJ} and  $v\in LSC(\R^N)\cap L^\infty(\R^N)$ is a viscosity supersolution of \eqref{eq:nHJ}, 
then $u\leq v \text{ in } \R^N$.
\end{theorem}

\begin{proof}
The proof is literally the same, at a sole modification in Lemma \ref{lem.est-NL-crown}. 
More precisely, on $B\setminus B_\delta$, we have that
\begin{eqnarray*}
\I[B\setminus B_\delta](\ovl{x}, \ovl p, \ovl{u}) - 
	    \I[B\setminus B_\delta](\ovl y, \ovl p + \ovl q, v)
& \le  & \frac{5}{\eps^2}  \int_{B\times B} |z_1 - z_2|^2 d\gamma(z_1, z_2) + 
         \frac{1}{\eps^2} o_{\delta}^{\ovl x,\ovl y}(1) + o_\beta(1)\\
& \le  & \frac{5}{\eps^2} 2^{2-p} \int_{B\times B} |z_1 - z_2|^p d\gamma(z_1, z_2) + 
        \frac{1}{\eps^2} o_{\delta}^{\ovl x,\ovl y}(1) + o_\beta(1).
\end{eqnarray*}
Taking infimum in $\gamma$ and employing the new assumption (M4)", we conclude that
\begin{eqnarray}
  \I[B\setminus B_\delta](\ovl{x},\ovl p,\ovl{u}) -
  \I[B\setminus B_\delta](\ovl y,\ovl p+\ovl q,v) 
&\leq& \frac{5}{\eps^2} 2^{2-p}  W_p^p(\nu_{\ovl{x}}, \nu_{\ovl{y}} )(B)
            +  \frac{1}{\eps^2} o_{\delta}^{\ovl x,\ovl y}(1)
    +o_\beta(1)  \nonumber \\
&\leq& \tilde C_p \frac{|\bar x-\bar y|^p}{\epsilon^2}+
    \frac{1}{\eps^2} o_{\delta}^{\ovl x,\ovl y}(1)
    +o_\beta(1).
     \label{eq:est-NL-crown-Wp}
\end{eqnarray} 

Moreover, we still have the boundedness of $\bar p$ as in Lemma~\ref{lem.p-bound}. Indeed, it follows from the viscosity inequality \eqref{eq:visc-ineq},  and the estimates provided in  \eqref{eq:est-fct},  \eqref{eq:est-H}, \eqref{eq:est-NL-in},  \eqref{eq:est-NL-out-unif},  \eqref{eq:est-NL-crown-Wp}, that 
\begin{eqnarray*} 
 \lambda \frac{M}{16} + \frac12(1 - \mu)b_m |\ovl{p} |^m  
 & \leq &  \tilde C_p |\ovl p| \;|\ovl x - \ovl y|^{p-1} + 
    \frac{1}{\eps^2}  o_\delta^{\ovl x,\ovl y}(1)  +  
    o_{\eps}(1) + o_{\beta}(1).
 \end{eqnarray*} 
Using that  $|\ovl x - \ovl y| = o_\eps(1)$ and $p\geq 1$, the conclusion of Lemma \ref{lem.p-bound} follows.

We then reiterate the estimate above, and use that $\bar p$ is bounded, to obtain
\begin{eqnarray*} 
 \lambda \frac{M}{16} 
& \leq  &  C |\ovl x - \ovl y|^{p-1}  + 
    \frac{1}{\eps^2}  o_\delta^{\ovl x,\ovl y}(1)  +  
    o_{\eps}(1) + o_{\beta}(1).
 \end{eqnarray*} 
Letting $\delta\to 0$, then $\eps,\beta\to 0$, and recalling that $p>1$, we arrive again at a contradiction.
\end{proof}

\section{Examples}\label{sec:examples}

In this section we give examples of nonlocal operators and PIDEs that fit into our framework, and for which comparison results hold.
We first detail different forms that L\'evy operators can take in order to satisfy assumptions (M1)--(M4). We then present a series of the PIDEs involving these operators for which we have uniqueness (and existence) results.

\subsection{Measures with density}\label{ex:kernel}

Let $(\nu_\xi)_{\xi \in\R^N}$ be a family of Lévy measures with density kernels:
$$\nu_{x}(dz) = K(x,z)dz,$$ where
$K : \R^N \times \R^N \rightarrow \R_+$ is a Carathéodory function. The nonlocal Lévy operator then writes 
\begin{equation*}
  \I_\xi u(x) = \int_{\R^N} \left( u(x + z) - u(x) - \pmb{1}_B(z) Du(x) \cdot z \right) K(\xi, z) dz.
\end{equation*}
We assume the kernel is degenerate elliptic of order $\sigma$, and Lipschitz continuous with respect to $\xi$. More precisely, let $K$ satisfy the following: 
\begin{itemize}
\item[(K1)] (Degenerate ellipticity)  
		There exist $\Lambda>0$ and $\sigma\in(0,2)$ such that, for all $x,z \in\R^N$,
		\[ 0 \leq K(x, z) \leq \frac{\Lambda}{|z|^{N+\sigma}}. \]
\item[(K2)] (Uniform Regularity)  
		There exist  $C_K>0$ and $\sigma\in(0,2)$ such that, for all $x,y, z \in\R^N$,
		\[| K(x, z) - K(y, z) | \leq C_K  \frac{|x-y|}{|z|^{N+\sigma}}. \]		
\end{itemize}	
In particular, if $K(x, z)= C(N,\sigma) k(x)  |z|^{-(N+\sigma)}$, 
with $k$ a Lipschitz continuous positive function, and $C(N,\sigma)$ the constant  appearing in (K2), then 
\[ \mathcal{I}_{x}u(x)= k(x)\left((-\Delta)^{\sigma/2}u \right)(x)\] 
is the $k-$weighted fractional Laplacian of order $\sigma\in(0,2)$. \smallskip

It is straightforward to check, in view of assumption (K1), that the family $(\nu_\xi)_\xi$ satisfies the uniform Lévy condition (M1), as well as the uniform decay at infinity condition (M2).
\begin{itemize}
\item[(M1)] The constant bounding uniformly the integrals is expressed in terms of $C(N,\sigma)$:
		\begin{eqnarray*}
		 \sup_{\xi \in \R^N}\int_{\R^N}\min(1,|z|^2) \nu_{\xi}(dz) 
			&\le& \int_{\R^N}\min(1,|z|^2)\frac{\Lambda}{|z|^{N+\sigma}} dz
			\;=\;\Lambda{\rm vol}(\partial B)\left(\frac{1}{2-\sigma}+\frac{1}{\sigma} \right).
		 \end{eqnarray*}
\item[(M2)] Similarly, the uniform decay at infinity is related to that of the fractional Laplacian:
		\begin{eqnarray*}
		 \sup_{\xi \in \R^N} \int_{B_R^c}  \nu_{\xi}(dz) 
		 	 =   \sup_{\xi \in \R^N} \int_{B_R^c}  K(\xi,z )dz
			\le  \int_{B_R^c} \frac{\Lambda}{|z|^{N+\sigma}} dz
            \; = \; \frac{\Lambda}{\sigma}{\rm vol}(\partial B)R^{-\sigma} 
			= \Lambda o_{1/R}(1).
		 \end{eqnarray*}
                
\end{itemize}	
We show below that, in view of condition (K2),  assumptions (M3) and (M4) are satisfied.
\begin{itemize}
\item[(M3)] Let {$0<r<R$}. Then, for any $x,y\in\R^N$, 
		\begin{eqnarray*} 
			|\nu_x-\nu_y|\left(B_R\setminus B_r\right) 
			& = & \int_{B_R\setminus B_r}|K(x,z) - K(y,z)| dz
			\;\le\; C_K|x-y|\int_{B_R\setminus B_r}\frac{1}{|z|^{N+\sigma}} dz \\
            & = & \frac{C_K}{\sigma} {\rm vol}(\partial B)\big(r^{-\sigma}-R^{-\sigma}\big)|x-y|
             = \omega_{R,r}(|x-y|).
		\end{eqnarray*}
\item[(M4)] Let $0<r<1$. For any $x,y\in\R^N$, we estimate the Wasserstein distance 
between L\'evy measures $\nu_x$ and $\nu_y$.
Using  Proposition~\ref{prop.Wass-TV} and (K2), for all $x,y\in\R^N$,
\begin{eqnarray*} 
W_2^2(\nu_x,\nu_y)(B_r) &\le&	\int \limits_{B_r}|z|^2|\nu_x-\nu_y|(dz)
			 =  \int_{B_r} |z|^2 |K(x,z) - K(y,z)| dz\\
			&\leq &  C_K  |x-y| \int_{B_r} \frac{|z|^2}{|z|^{N+\sigma}} dz
             =  \frac{C_K}{2-\sigma}{\rm vol}(\partial B)  r^{2-\sigma} |x-y|
             =  o_r(1) |x-y|.
\end{eqnarray*}

\end{itemize}
		
\subsection{Nonlocal operators of variable order}\label{ex:var-order}

Let $(\nu_\xi)_{\xi \in\R^N}$ be a family of Lévy measures with density kernels $(K(\xi,\cdot))_{\xi}$ 
of variable order, given by
\[ K(\xi,z) = \frac{1}{|z|^{N+\sigma(\xi)}}, \] 
where $\sigma: \R^N \rightarrow (0, 2)$ is a function satisfying: 
\begin{itemize}
\item[(S1)] (Uniform ellipticity)  
		There exist $0 < \sigma_1 < \sigma_2 < 2$ such that, for all $x \in\R^N$,
$ \sigma_1 \leq \sigma(x) \leq \sigma_2.$
\item[(S2)] (Lipschitz Regularity)  
  There exists $C_\sigma>0$  such that, for all $x,y \in\R^N$,
  $| \sigma(x) - \sigma(y)| \leq C_\sigma |x-y|.$
\end{itemize}	

The nonlocal Lévy operator then writes 
\begin{equation*}
  \I_\xi u(x) = \int_{\R^N} \left[ u(x + z) - u(x) - \pmb{1}_B(z) Du(x) \cdot z \right]  \frac{1}{|z|^{N+\sigma(\xi)}} dz.
\end{equation*}

Note that the kernel satisfies, for all $z\in\R^N$,
\[  \min \left( \frac{1}{|z|^{N+\sigma_1}}, \frac{1}{|z|^{N+\sigma_2}}\right)
	\leq \frac{1}{|z|^{N+\sigma(\xi)}} \leq 
   \max \left( \frac{1}{|z|^{N+\sigma_1}}, \frac{1}{|z|^{N+\sigma_2}}\right).\]
In view of this estimate, we can easily check that assumptions (M1) and (M2) are satisfied.

\begin{itemize}
\item[(M1)] The constant bounding uniformly the integrals is expressed in terms of $\sigma_1$ and $\sigma_2$:
\begin{eqnarray*}
		 \sup_{\xi\in\R^N} \int_{\R^N} \min(1,|z|^2)\nu_{\xi}(dz) 
		 	&\le& \sup_{\xi \in \R^N} \int_{B}|z|^2 \frac{1}{|z|^{N+\sigma(\xi)}}dz  +
				  \sup_{\xi \in \R^N} \int_{B^c} \frac{1}{|z|^{N+\sigma(\xi)}}dz\\
		 	&\le& \int_{B} |z|^2\frac{1}{|z|^{N+\sigma_2}}dz  
                + \int_{B^c}\frac{1}{|z|^{N+\sigma_1}}dz
             = {\rm vol}(\partial B)\left(\frac{1}{2-\sigma_2}+\frac{1}{\sigma_1} \right).
\end{eqnarray*}

\item[(M2)] Similarly, the uniform decay at infinity can be expressed in terms of a function
		 associated with the fractional Laplacian of order $\sigma_1$. Indeed, for any $R>1$, we have
\begin{eqnarray*}
		 \sup_{\xi \in \R^N} \int_{B_R^c}  \nu_{\xi}(dz) 
		 	 =   \sup_{\xi \in \R^N} \int_{B_R^c}  \frac{1}{|z|^{N+\sigma(\xi)}}dz
			\le  \int_{B_R^c}  \frac{1}{|z|^{N+\sigma_1}} dz
            =  \frac{1}{\sigma_1}  {\rm vol}(\partial B) R^{-\sigma} 
			\;=\; o_{1/R}(1).
\end{eqnarray*}
		
\end{itemize}	

In order to show that (S2) implies (M3) and (M4), we rely on the following remark. Fix $x, y, z \in \R^N$ and without the loss of generality, we assume $\sigma(x) < \sigma(y)$. Then, by the mean value theorem applied to the real function $t\mapsto |z|^t$, it follows that there exists 
$\xi = \xi(x, y) \in (\sigma(x),\sigma(y))\subset (\sigma_1,\sigma_2)$ such that 
\begin{eqnarray}\label{eq:sigma_x} 
\left| |z|^{\sigma(x)} - |z|^{\sigma(y)} \right|
	& = & |z|^{\xi} \ln(|z|) \left|\sigma(y) - \sigma(x)\right| 
	\le C_\sigma |z|^{\xi}  |\ln(|z|)| |x-y|. 
\end{eqnarray}	
This is sufficient to get the estimates required by (M3) and (M4), as follows.

\begin{itemize}
\item[(M3)] Let {$0<r<1<R$}. Then, for any $x,y\in\R^N$, it holds that
		\begin{eqnarray*} 
			|\nu_x-\nu_y|\left(B_R\setminus B_r\right) 
			& = & \int_{B_R\setminus B_r}   \left| \frac{1}{|z|^{N+\sigma(x)}} - \frac{1}{|z|^{N+\sigma(y)}}\right| dz
			 =  \int_{B_R\setminus B_r}  \frac{\left|| z|^{\sigma(x)} -|z|^{\sigma(y)}\right|} {|z|^{N+\sigma(x) + \sigma(y)}} dz\\
			& \leq &  C_\sigma  |x - y|
				\int_{B_R\setminus B_r}  \frac{ |z|^{\xi}\big|\ln(|z|)\big|}{|z|^{N+\sigma(x)+\sigma(y)}} dz \\
			& \leq &  C_\sigma  |x - y| \left(
				\int_{B_R\setminus B}  \frac{\big|\ln(|z|) \big|}{|z|^{N+\sigma_1}} dz + 
				\int_{B\setminus B_r}  \frac{\big|\ln(|z|) \big|}{|z|^{N+\sigma_2}} dz \right).
		\end{eqnarray*}	
		We now evaluate the last integral terms in the sum and get
		\begin{eqnarray*} 
		\int_{B_R\setminus B} \frac{\big|\ln(|z|) \big|}{|z|^{N+\sigma_1}} dz
			 & = & {\rm vol}(\partial B)\int_1^R \frac{\ln(s)}{s^{1+\sigma_1}} ds
			   =   {\rm vol}(\partial B)\left(-\dfrac{R^{-\sigma_1}}{\sigma_1} 
			 	   \left(\ln(R)+\dfrac{1}{\sigma_1} \right)+\dfrac{1}{\sigma_1^2}\right),\\
		\int_{B\setminus B_r} \frac{\big|\ln(|z|) \big|}{|z|^{N+\sigma_2}} dz
			 & = & {\rm vol}(\partial B) \int_r^1  \frac{-\ln(s)} {s^{1 + \sigma_1}} ds 
			   =   {\rm vol}(\partial B)\left(-\dfrac{r^{-\sigma_2}}{\sigma_2} 
			 	   \left(\ln(r)+\dfrac{1}{\sigma_2}\right)+\dfrac{1}{\sigma_2^2}\right). 	
		\end{eqnarray*}	
		Summing up we obtain a Lipschitz bound for the total variation between the two Lévy measures, on the circular crown $B_R \setminus B_r$.
\item[(M4)] Let $0<r<1$. In view of Proposition \ref{prop.Wass-TV}
        and of inequality \eqref{eq:sigma_x}, for any $x,y\in\R^N$, it holds 
 \begin{eqnarray*} 
		W_2^2(\nu_x,\nu_y)(B_r) 
			& \le & \int \limits_{B_r}|z|^2|\nu_x-\nu_y|(dz)
			 =  \int_{B_r} \frac{|z|^2 \left|| z|^{\sigma(x)} -|z|^{\sigma(y)}\right|}
							    {|z|^{N+\sigma(x) + \sigma(y)}} dz\\
			&\leq &  C_\sigma|x-y|\int_{B_r}\frac{|z|^{2+\xi}\big|\ln(|z|)\big|}                                                 {|z|^{N+\sigma(x)+\sigma(y)}} dz
			\leq  C_\sigma|x-y|\int_{B_r}\frac{\big|\ln(|z|) \big|}
							                     {|z|^{N+\sigma_2- 2}} dz
                = : \omega(r)|x - y|, 
\end{eqnarray*}	
		where $\omega: \R_+ \to \R_+$ is given by  
		\begin{eqnarray*} 
		\omega(r) = C_\sigma\int_{B_r} \frac{\big|\ln(|z|) \big| }{|z|^{N+\sigma_2- 2}}dz
			 & = & C_\sigma {\rm vol}(\partial B)\dfrac{r^{2 - \sigma_2}}{2 - \sigma_2} 
			 		\left( \dfrac{1}{2 - \sigma_2} - \ln(r) \right). 
		\end{eqnarray*}			 
\end{itemize}

\subsection{Lévy-Itô operators}\label{ex:Levy-Ito}

We focus next on Lévy-Itô operators, which take the form
\begin{equation}\label{eq:Levy-Ito}
 	\mathcal J_{\xi} u(x) = \int_{\R^N} \left(u(x + {j(\xi,z)}) - u(x) - \pmb{1}_B(z) Du(x) \cdot {j(\xi,z)} \right) \nu(dz),
\end{equation}
where $\nu$ is a Lévy measure that is absolutely continuous with respect to the Lebesgue measure, and the jump function $j:\R^N\times\R^N\to\R^N $ satisfies  the following set of assumptions:
\begin{itemize}
\item[(J1)] For any $\xi\in\R^N$, the jump function at $\xi$, given by $j(\xi, \cdot):\R^N\to\R^N$, is invertible.
\item[(J2)] There exist  $0 <c_0 \leq c_1$ such that, for all $z \in\R^N$,
		\[ c_0 |z| \leq |j(\xi,z)|\leq c_1 |z|.\]
\item[(J3)] For any $r>0$, there exists a modulus of continuity 
		$\omega_{r}:\R_+\to\R_+$ such that for all $x,y\in\R^N$,
        and for all $z\in\R^N\setminus B_r$,
		\[|j(x,z) - j(y,z)| \leq \omega_{r}(|x - y|).\]
\item[(J4)] There exists $r_0>0$ such that, for all $r\in(0,r_0)$, and for all $x,y\in\R^N$,
		\[\int_{B_r} |j(x,z) - j(y,z)|^2 \nu(dz) \leq o_r(1) |x - y|.\] 
\end{itemize}		

These operators naturally appear as infinitesimal generators of non-symmetric stochastic jump processes, see~\cite{OS05}. Comparison principles for nonlocal Lévy operators have been first established in Lévy-Itô form, by Barles and Imbert in \cite{BI08} (see also \cite{JK06}). In their paper, no invertibility assumption is required, due to the fact that comparison principles are proved directly for nonlocal operators in Lévy-Itô form, and the proof takes into account their special structure. In our case, the passage from \eqref{eq:Levy-Ito} to the above operator requires, at this level, the invertibility of the jump function. We prove that under assumptions (J1)-(J4), Lévy-Itô operators can be reformulated as Lévy operators, for which the comparison result stated in Theorem \ref{thm.comparison} holds.

The results of Barles and Imbert in \cite{BI08} are more general than the ones we present here. 
In particular, there is no constraint with respect to the invertibibility of the jump function.
The interest of our approach is that it gives a comparison result for nonlocal operators with a slightly less general assumption (J4) than the corresponding part in \cite{BI08}.

\subsubsection{Lévy-Itô operators through the lens of general Lévy operators} 
We start by making precise how Lévy-Itô operators can be written in the general Lévy form,  through an associated \textsl{push-forward measure}. To this end, consider, for any point $\xi\in\R^N$, the jump function at $\xi$, given by $j_\xi:= j(\xi,\cdot):\R^N\to\R^N$ and let the push forward measure of the Lévy measure $\nu$ through ${j_\xi}$ be given by
\[ \nu_\xi^{j} := ({j_\xi})_\# \nu. \]
The above measure can also be defined by duality, through integration against measurable functions, as follows.  For each measurable function $f: \R^N \to \R$ such that $f(z)\leq C\min (1, |z|^2)$, let
\begin{eqnarray*}
	\int_{\R^N} f(z) \nu_\xi^{j}(dz) & = &\int_{\R^N} f({j_\xi(z)}) \nu (dz).
\end{eqnarray*}
In view of (J2), the measure $\nu_\xi^{j}$ is well defined and it satisfies the  Lévy condition.\smallskip

In view of assumption $(J1)$, Lévy-Itô operators given by \eqref{eq:Levy-Ito} can be re-written, through a change of variables, as
\[ \mathcal {J}_\xi u(x) = \int_{\R^N} \big( u(x + z) - u(x) - \1_{j_\xi(B)}(z)  Du(x) \cdot z \big) \nu_\xi^j(dz). \]
The term $\1_{j_x(B)}(z) Du(x) \cdot z$ may fail, a priori, to compensate the first-order finite difference of $u$ at $x$ if the origin is not an interior point of ${j_\xi}(B)$. This is, however, ensured by the uniform linear control given by assumption (J2), which ensures that $B_{rc_0}\subset {j_\xi}(B_{r})$, for all $r>0$. Therefore, we can write the nonlocal Lévy-Itô operator as
\[  \mathcal J_\xi u(x) = \I_\xi u(x) + b^j(\xi) \cdot Du(x),\]
where $\I_\xi $ is a nonlocal operator taking the Lévy form
\[ \I_\xi u(x) = \int_{\R^N} \big(u(x + z) - u(x) - \1_{B}(z) Du(x)\cdot z\big) \nu_\xi^j(dz),\]
and $b^j(\xi)$ is the vector field given  by
\begin{equation*}
b^j(\xi):=  
    \int_{\R^N}\left(\1_{B}(z)-\1_{{j_\xi}(B)}(z)\right)\;z\;\nu_\xi^j(dz) 
 =  \int_{\R^N}\left(\1_{j_\xi^{-1}(B)}(z)-\1_{B}(z)\right)\;j_\xi(z)\;\nu(dz).
\end{equation*}
We show below the vector field is well defined and bounded, and the family of measures $(\nu_\xi^j)_\xi$ satisfies assumptions (M1)-(M4).

\begin{remark}[\em Symmetric Lévy-Itô operators]\em
 In the case of symmetric Lévy measure, i.e.,  when $\nu(-A) = \nu(A)$, for all measurable sets $A\subset\R^N$,
assumption (J1) is not necessary anymore, as we assume instead the following:  
\begin{itemize}
    \item[(J1)'] For any $\xi\in\R^N$, the jump function ${j_\xi}:\R^n\to\R^N$ is symmetric, i.e., ${j_\xi}(-z) ={j_\xi(z)}$ for all $z\in\R^N$.
\end{itemize}
Then, the Lévy-Itô operators can be written as 
\[  \mathcal J_{\xi} u(x) = \mathrm{P.V.} \int_{\R^N} \big( u(x + j(x,z)) - u(x) \big) \nu(dz)
   = \mathrm{P.V.} \int_{\R^N} \big( u(x + z) - u(x)\big) \nu_\xi^j(dz). \]
It is easy to see that the operator is well-defined, and we shall check below that assumptions (M1)-(M4) are satisfied.
\end{remark}

\noindent \subsubsection{Boundedness and regularity of the vector field $b^j(\cdot)$}
The extra drift term will be considered as part of the Hamiltonian, posing 
$$\widetilde H(x,p) = H(x,p) + b^j(x) \cdot p.$$
Thus, running the comparison proof with $\widetilde H$ instead of $H$, it requires for $\widetilde H$ to satisfy the assumptions (H0)-(H2). The map $p \mapsto b^j(x) \cdot p$ is linear, and therefore it will not affect the boundedness assumption (H0). However, the superlinear coercivity condition (H2) relies on the boundedness of $b^j_\xi$, whereas the regularity assumption (H1) requires a modulus of continuity for $b^j_\xi$. Note first that, in view of the invertibility assumption (J1) and of the uniform bounds (J2), we have, for any $r>0$, and for all $\xi \in\R^N$,
\begin{equation}\label{eq:ball-inclusion-LI}
     B_{rc_0} \subseteq j_\xi(B_r)      \subseteq B_{rc_1}\quad \text{and}\quad
     B_{r/c_1}\subseteq j_\xi^{-1}(B_r) \subseteq B_{r/c_0}.
\end{equation}

On one hand, in view of assumptions (J1) and (J2), it follows that the vector field is well-defined and bounded:
\begin{eqnarray*}
| b^j(\xi) | 
    &\le & \int_{\R^N} \left|\1_{j_\xi^{-1}B}(z)-\1_{B}(z)\right|\;|j_\xi(z)|\ \; \nu(dz) 
    \;\le\;c_1 \int_{B\Delta j_\xi^{-1}(B)}| z| \nu(dz) \\
    &\le & c_1 \max(1,\frac{1}{c_0}) \int_{B\Delta j_\xi^{-1}(B)}\nu(dz)
    \;\le\;c_1 \max(1, \frac{1}{c_0})\int_{B_{\max\{c_0^{-1},1\}}\setminus B_{\min\{c_1^{-1}, 1\}}}\nu(dz) = : C_{c_0,c_1,\nu},
\end{eqnarray*}
with $C_{c_0,c_1,\nu}$ a constant depending only on $c_0,c_1$ and $C_\nu$.

On the other hand, we obtain a modulus of continuity for the vector field. Taking into account the ball inclusion \eqref{eq:ball-inclusion-LI}, there exists $r_0\leq \min(1, c_0)$ and $r_1=\min(1, r_0 / c_1)$ such that, uniformly in $\xi$,
\[ B\setminus j_\xi(B) = (B\setminus B_{r_0})\ \setminus j_\xi(B) 
= (B\setminus B_{r_0})\setminus j_\xi(B\setminus B_{r_0}).\]
Hence, the following estimate holds
\begin{eqnarray*}
|b^j(x) - b^j(y)| 
	& = &  \left | \int_{B\setminus j_x(B)} z\nu_x^j(dz)  - 
		   \int_{B\setminus j_y(B)} z\nu_y^j(dz)   \right| \\
	& = &  \left | \int_{B\setminus B_{r_0}} z \left(\nu_x^j- \nu_y^j\right)(dz) + 
		   \int_{B\setminus B_{r_1}} \left( j(x,z) - j(y,z)\right)\nu(dz)  \right| \\
	& \leq &	\int_{B\setminus B_{r_0}} |z| |\nu_x^j - \nu_y^j | (dz)+ 
		 \int_{B\setminus B_{r_1}} \left| j(x,z) - j(y,z) \right| \nu(dz)\\
   	& \leq & r_0  |\nu_x^j - \nu_y^j | (B\setminus B_{r_0})+  \int_{B\setminus B_{r_1}} \left| j(x,z) - j(y,z) \right| \nu(dz).
\end{eqnarray*}
Similarly to the computations performed below for $(M3)$,  it follows that
\begin{eqnarray*}
|b^j(x) - b^j(y)|  & \le & \widetilde \omega_{1,r_0}(|x-y|) +  \omega_{1,r_1}(|x-y|).
\end{eqnarray*}

\subsubsection{Assumptions (M1)-(M4) for the family of push-foward Lévy measures}
We show below that the family of Lévy measures $(\nu^j_\xi)_{\xi\in\R^N}$ is well-defined, and that assumptions (M1)-(M4) are satisfied, provided (J1)-(J4) hold.

\begin{itemize}
\item[(M1)]  It follows immediately from (J2) and the definition of the push forward measure that
 \begin{eqnarray*}
	\sup_{\xi \in \R^N} \int_{\R^N} \min(1,|z|^2) \nu_{\xi}^j(dz) 
		&=  &\sup_{\xi \in \R^N} \int_{\R^N} \min ( 1, |{j_\xi(z)}|^2 ) \nu(dz) \\
		& \leq &c_1^2 \int_{B_{1/c_1}}|z|^2\nu(dz) +  \int_{B_{1/c_1}^c}\nu(dz)<\infty.
\end{eqnarray*} 

\item[(M2)] Let $R>1$. It follows from the definition of the push forward measure, assumption (J2) and the definition of the Lévy measure, that
		\begin{eqnarray*}
		\sup_{\xi \in \R^N} \int_{B_R^c} \nu_{\xi}^j(dz) 
			& = & \sup_{\xi \in \R^N} \int_{({j_\xi})^{-1}(B_R^c)} \nu(dz)
			\le  \int_{B_{R/c_1}^c} \nu(dz)  = o_{1/R}(1).
		\end{eqnarray*} 		
\item[(M3)] Let $0<r<R$. Notice that, for all $x,y\in\R^N$, 
		$| \nu_x^j - \nu_y^j| = \left(\nu_x^j - \nu_y^j\right)^+ + 
				   \left(\nu_x^j - \nu_y^j\right)^-$
        and
\begin{eqnarray*}
		\int_{B_R\setminus B_r}  \left(\nu_x^j - \nu_y^j\right)^+(dz)  
			& = & \sup_{A\subset B_R\setminus B_r } 
                \left( \int_A \nu_x^j (dz) - \int_A\nu_y^j(dz)  \right)\\
			 & = & \sup_{A\subset B_R\setminus B_r } 
                \left( \int_{j_x^{-1}(A)} \nu(dz) - 
					   \int_{j_y^{-1}(A)} \nu(dz) \right)\\
			&\le&  \sup_{A\subset B_R\setminus B_r }
                \nu\left(j_x^{-1}(A) \setminus j_y^{-1}(A)\right),
\end{eqnarray*}
and we have a similar estimate for $\left(\nu_x^j - \nu_y^j\right)^-$.

In view of (J2), we have 
        $\displaystyle  j_x^{-1}(A) \setminus j_y^{-1}(A) \subseteq 
            j_{x}^{-1}(A) \subseteq j_{x}^{-1}(B_R\setminus B_r) \subseteq B_{R/c_0}\setminus B_{r/c_1}.$
        It follows, in the first place, that the total variation is finite:
		\begin{eqnarray*}
		| \nu_x^j - \nu_y^j| \left(B_R\setminus B_r\right) 
			& \leq & 2 \int_{B_{R/c_0} \setminus B_{r/c_1}} \nu(dz) <\infty.
		\end{eqnarray*} 	
        In view of assumption (J3), $j_y(z) = j_x(z) + \omega_{r/c_1}(|x-y|) e$,
        with $e\in\R^N$ a unit vector, depending on $z,x,y$ and $j$. In the worst case scenario, the set shrinks by a factor of $\omega_{r/c_1}(|x-y|)$, and we have a rough estimate 
        \begin{eqnarray*}
             | j_x^{-1}(A) \setminus j_y^{-1}(A)| \leq
            |A|\, \tilde \omega_r(|x-y|),
        \end{eqnarray*}    
        with $\widetilde \omega_r(s) = \left(\omega_{r/c_1}(s)\right)^N$. Finally, if the measure $\nu$ is absolutely continuous with respect to the Lebesgue measure, this amounts to
        \begin{eqnarray}\label{est942}
		| \nu_x^j - \nu_y^j| \left(B_R\setminus B_r\right) 
			& \leq & \widehat\omega_{R,r}(|x-y|).
		\end{eqnarray}
\item[(M4)] Let  $r_0 = \min(1, 1/c_1)$ and $r\in(0,r_0)$. Then, for any $x,y\in\R^N$, the Wasserstein distance 
		between the L\'evy measures $\nu_x^j$ and $\nu_y^j$ restricted on $B_r$ is given by
		\begin{eqnarray*}
		W_2^2(\nu_x^j,\nu_y^j)(B_r) & = &			
		 \inf_{\gamma} \int \limits_{B_r\times B_r} |z_1 - z_2|^2 d\gamma(z_1,z_2).
		\end{eqnarray*} 	
		Let  $\gamma = (j_x \times j_y)_\# \nu\mid_{B_r}$.
                We claim that $\gamma$ is an admissible coupling plan for $\nu^j_x\mid_{B_r}$ and $\nu^j_y\mid_{B_r}$. 
                Indeed, in view of Lemma \ref{lem.supp-gamma}, $\text{supp}(\gamma) \subset B_r\times B_r$, 
                and hence, for all $A\subset B_r\setminus\{0\}$ measurable, we have
		\begin{eqnarray*}
		\pi^1_\# \gamma (A) 
		& = & \gamma((\pi^1)^{-1}(A)) = \gamma(A\times B)
		  =  (j_x \times j_y)_\# \nu\mid_{B_r}(A\times B)\\
		& = & \nu\mid_{B_r} \Big(\big(j_x \times j_y\big)^{-1}(A\times B)\Big)
		  =   \nu\mid_{B_r} \Big(\big(j_x\big)^{-1}(A)\cap\big(j_y\big)^{-1}(B)\Big) \\
		& = & \nu \Big(\big(j_x\big)^{-1}(A)\Big)
		  =  (j_x)_\# \nu(A) = \nu_x^j\mid_{B_r}(A),
		\end{eqnarray*} 	
		using assumption (J2), which ensures that $B_r\subseteq B_{1/c_1} \subseteq \big(j_y\big)^{-1}(B)$.
 		Similarly, $\pi^2_\# \gamma (A)=\nu_y^j\mid_{B_r}(A)$ and the claim is proved.
                
		Now, for each $r \in (0,r_0)$ we have 
		\begin{eqnarray*}
			\int \limits_{B_r\times B_r} |z_1 - z_2|^2 d\gamma(z_1,z_2) 
				&= & \int \limits_{B_r\times B_r} |z_1 - z_2|^2 d(j_x \times j_y)_\# \nu(z_1, z_2) \\
				&= & \int_{j_x^{-1}(B_r) \cap j_y^{-1}(B_r)} |j(x,z) - j(y,z)|^2 \nu(dz) \\
				&\le & \int_{B_{r/c_0}} |j(x,z) - j(y,z)|^2 \nu(dz).
		\end{eqnarray*}
	The conclusion then follows from assumption $(J4)$.
	\end{itemize}
\appendix 
\section{Nonlocal estimate for a localization function}\label{sec:appendix-NL}

We present some fine estimates of nonlocal operators acting on smooth localization functions, extending some results previously obtained in ~\cite{BLT17}. These are used for PIDEs posed on unbounded domains, where we employ a localization function, which can take the following form.
Let $c>0$ be a fixed constant. Consider $\psi \in C^2(\R^N)\cap L^\infty(\R^N)$ such that
	\begin{equation}\label{psi_func}
		\begin{dcases}
			\psi = 0 & \mbox{in } B, \\
			\psi = c & \mbox{in } \R^N \setminus B_2, \\
			0 \leq \psi \leq c & \mbox{in } B_2 \setminus B.
		\end{dcases}
	\end{equation}
Notice that there exists a constant $C_0>0$ such that 
\begin{equation}\label{eq:C0}
  ||  \psi || _{\infty}, ||  D\psi || _{\infty}, || D^2 \psi ||_{\infty} \leq C_0.
\end{equation}
The following lemma is an extension of~\cite[Lemma 2.3]{BLT17}  to the case of a family of point-dependent L\'evy measures $(\nu_\xi)_{\xi\in\R^N}$, satisfying the uniform integrability assumption  (M1).

\begin{lemma} \label{lem.localization}
Let $(\nu_\xi)_{\xi\in\R^N}$ be a family of L\'evy measures satisfying assumption (M1).
Let $\psi$ be defined as in \eqref{psi_func}. Then, for  any $\beta \in (0, 1)$,
the function $\psi_{\beta}(x) := \psi(\beta x)$  satisfies
\begin{eqnarray}
       \nonumber && 
       	||  D\psi_{\beta} || {\infty} \leq C_0 \beta, \, || D^2 \psi_{\beta} ||_{\infty} \leq C_0 \beta^2, \\
       \nonumber && 
       	\sup_{\xi\in \R^N}\I_\xi[B_\delta](x, \psi_{\beta}) \leq  
		\min \left(\frac12 C_0 C_\nu \beta^2,  \beta^2 o_\delta^\xi(1) \right),\\
      \label{eq:unif-est-beta}  && 
      	\sup_{\xi\in \R^N} \I_\xi[B_\delta^c](x, \psi_{\beta}) \leq o_{\beta}^\xi(1),
\end{eqnarray}
where $C_0$ appears in~\eqref{eq:C0} and $C_\nu$ in (M1).
If, in addition (M2) holds, then estimate \eqref{eq:unif-est-beta} holds uniformly  with respect to $\xi\in\R^N$.
\end{lemma}

\begin{proof}The estimates for $D\psi_\beta$ and $D^2 \psi_\beta$ are obvious.
In order to estimate the nonlocal term on $B_\delta$, we take into account the $C^2$ regularity of  $\psi_\beta$, and write
\begin{eqnarray*}
    \I_\xi[B_\delta](x, \psi_{\beta}) = 
    	\frac{1}{2} \int_{|z|\leq \delta}\int_0^1 D^2\psi_\beta(x+\theta z)z\cdot z d \theta\, \nu_\xi (dz).
\end{eqnarray*}
A direct estimate gives, in view of assumption (M1),
\begin{eqnarray*}
|\I_\xi[B_\delta](x, \psi_{\beta})| 
&\le &\frac12| D^2\psi_\beta||_{\infty} \int_{|z|\leq \delta} |z|^2 \nu_\xi (dz)
\;\le\;\frac12 C_0 \beta^2 \int_{|z|\leq \delta} |z|^2 \nu_\xi (dz)
	 \leq  \frac{1}{2} C_0 C_\nu \beta^2.
\end{eqnarray*}
Moreover, for a fixed $\xi\in\R^N$, the regularity of the measure $\nu_\xi$ gives
$ \displaystyle \lim_{\delta \rightarrow 0} \int_{|z|\leq \delta} |z|^2 \nu_\xi (dz) = 0, $
which, in turn, implies the desired bound
\begin{eqnarray*}
|\I_\xi[B_\delta](x, \psi_{\beta})|\leq \frac{1}{2} \beta^2 o_\delta^\xi (1).
\end{eqnarray*}

In order to estimate the nonlocal term on $B_\delta^c$, we write
\begin{align*}
    & \I_\xi[B_\delta^c](x, \psi_{\beta}) =  
    ~ \I_\xi[B](x, \psi_{\beta}) - \I_\xi[B_\delta](x, \psi_{\beta}) + \I_\xi[B^c](x,  \psi_{\beta}).
\end{align*}
From the previous step, it follows that
\begin{eqnarray*}
\left| \I_\xi[B](x, \psi_{\beta}) \right| + \left| \I_\xi[B_\delta](x, \psi_{\beta}) \right|\leq C \beta^2.
\end{eqnarray*}
To conclude, it remains to prove that $\I_\xi[B^c](x, D\psi_{\beta}(x),  \psi_{\beta})\to 0$
as $\beta\to 0$, uniformly with respect to $x,\xi$. For any $R\geq 1$, the nonlocal term writes
\begin{eqnarray*}
&& \I_\xi[B^c](x,  \psi_{\beta}) =  \I_\xi[B_R\setminus B](x,  \psi_{\beta}) +  \I_\xi[B_R^c](x,  \psi_{\beta}).
\end{eqnarray*}
In order to estimate the first term, we rely on the regularity of $\psi_\beta$ and (M1), whereas for the second one, on the regularity of the measure. Thus, it follows that
\begin{eqnarray*}
\I_\xi[B_R\setminus B](x,  \psi_{\beta}) 
& = & \int_{1\leq |z|\leq R}  [\psi_{\beta}(x + z) -  \psi_{\beta}(x) ] \nu_{\xi}(dz)
\;=\; \int_{1\leq |z|\leq R}  \int_0^1 D\psi_{\beta}(x + \theta z) \cdot z d\theta \nu_{\xi}(dz) \\
&\le&  |  D\psi_{\beta} | _\infty  \int_{1\leq |z|\leq R}  |  z|    \nu_{\xi}(dz)
\; \le \; \beta C_0 \int_{1\leq |z|\leq R}  |  z|  \nu_{\xi}(dz)
\; \le  \;\beta C_0 C_\nu R,  \\
\I_\xi[B_R^c](x,  \psi_{\beta}) 
& = &   \int_{|z|> R} [\psi_{\beta}(x + z) -  \psi_{\beta}(x) ] \nu_{\xi}(dz)
\;\le \;   2 |  \psi_\beta | _\infty \int_{|z|> R}  \nu_{\xi}(dz)
\;\le \;   2C_0  o_{1/R}^\xi(1).
\end{eqnarray*}
Therefore, there exists a constant $C>0$ such that, for all $R\ge 1$, it holds
\[  \I_\xi[B^c](x,  \psi_{\beta})\leq C_0\left(C_\nu \beta R +  o_{1/R}^\xi (1) \right). \]
In view of Lemma~\ref{lem.modul-cont} below, taking infimum over all $R\ge 1$, \eqref{eq:unif-est-beta} follows.

When assumption (M2) is in place, one can replace $o_{1/R}^\xi(1)$  in the  inequality  \eqref{eq:unif-est-beta} with $o_{1/R} (1)$, which tends to $0$ as $R\to +\infty$ uniformly with respect to $\xi$, leading to 
\[\sup_{\xi\in\R^N} \I_\xi[B^c](x,  \psi_{\beta}) \leq o_\beta(1).\]
\end{proof}

\begin{lemma} \label{lem.modul-cont}
Let $g: [1,\infty)\to [0,\infty)$ such that $g (R)\to 0$ as $R\to +\infty$. Define
$ \displaystyle f(\beta):=\inf_{R\ge 1}\left(g(R)+R\beta\right). $
Then $f(\beta)\to 0$ as $\beta\to 0$.  
\end{lemma}

\begin{proof}[Sketch of proof of Lemma~\ref{lem.modul-cont}]
Without loss of generality, we may assume that $g$ is a smooth, decreasing, convex function (otherwise one can replace $g$ with a larger function, which still tends to $0$ at $+\infty$).  In this case, the infimum in $R$ is achieved at some $R_\beta$ satisfying $g'(R_\beta)+\beta =0.$ Therefore
$R_\beta = (g')^{-1}(-\beta) \mathop{\to}_{\beta\to 0} +\infty. $
In view of the convexity of $g$, we have, for all $R\geq 1$,
\begin{eqnarray*}
g'(R_\beta) (R-R_\beta)\leq g (R)-g (R_\beta).
\end{eqnarray*}
Writing the previous inequality for $R=R_\beta/2$ and using the expression for $R_\beta$,  we obtain
\begin{eqnarray*}
\beta R_\beta \leq 2\left(  g (R_\beta/2)-g (R_\beta) \right)
 \mathop{\to}_{\beta\to 0} 0.  
\end{eqnarray*}
We conclude that $f(\beta)= g (R_\beta) +\beta R_\beta \to 0$
as $\beta \to 0$.
\end{proof}

\bigskip\paragraph*{\textbf{Acknowledgments}}

 A. Ciomaga is partially supported by the ANR (Agence Nationale de la Recherche) through the COSS project ANR-22-CE40-0010 and by the Research Grant GAR2023 - code 73 - from the Donors' Recurrent Fund, at the disposal of the Romanian Academy and managed through the "PATRIMONIU" Foundation.
T. M. L\^e is supported by the Austrian Science Fund grant FWF P-36344N. 
O. Ley is partially supported by the ANR (Agence Nationale de la Recherche) through the COSS project ANR-22-CE40-0010 and the Centre Henri Lebesgue ANR-11-LABX-0020-01.
E. Topp is partially supported by CNPq Grant 306022/2023-0, CNPq Grant 408169/2023-0, and FAPERJ APQ1 Grant 210.573/2024.



\end{document}